\documentclass[twocolumn]{autarty}
\usepackage{graphicx}
\usepackage{comment}
\usepackage{arydshln}
\usepackage{mathtools}
\usepackage{graphics} 
\usepackage{epsfig} 
\usepackage{mathptmx} 
\usepackage{times} 
\usepackage{amsmath} 
\usepackage{amssymb}  
\usepackage{enumerate}
\usepackage{multirow}
\usepackage{latexsym}
\usepackage{bm}
\usepackage{dsfont}
\usepackage{multirow}
\usepackage{color}
\usepackage{url}
\usepackage{algorithm}
\usepackage{cite}
\usepackage{algpseudocode}
\newtheorem{remark}{Remark}
\newtheorem{assumption}{Assumption}
\newtheorem{lemma}{Lemma}
\newtheorem{proposition}{Proposition}
\newtheorem{definition}{Definition}
\newtheorem{corollary}{Corollary}
\newtheorem{theorem}{Theorem}

\newcommand{\zxl}[1]{{\color{black}{ #1}}}

{
	\begin{bmatrix}}%
	{\end{bmatrix}
	}

\def\build#1_#2^#3{\mathrel{\mathop{\kern0pt#1}\limits_{#2}^{#3}}}%

\def\min#1{\build{\rm min}_{#1}^{}}
\def\build#1_#2^#3{\mathrel{\mathop{\kern0pt#1}\limits_{#2}^{#3}}}%

\def\BibTeX{{\rm B\kern-.05em{\sc i\kern-.025em b}\kern-.08em
		T\kern-.1667em\lower.7ex\hbox{E}\kern-.125emX}}

\begin{document}
	
	\begin{frontmatter}
		
		\title {Robust Koopman MPC with Sets Updates for Time Delayed Systems$^{\star}$} 
		\vspace{-4mm}
		\thanks[footnoteinfo]{This paper was not presented at any 
			meeting.  } 
	\author[XLZ]{Xinglong Zhang}\ead{zhangxinglong18@nudt.edu.cn},    
	\author[XLZ]{Xinxin Yao}\ead{yxx528383145@163.com},               
	\author[XLZ]{Xin~Xu}\ead{xinxu@nudt.edu.cn},  
    \author[YKZ]{Keyou You}\ead{youky@tsinghua.edu.cn},  
    \author[XLZ]{Dewen~Hu}\ead{dwhu@nudt.edu.cn}  
	\address[XLZ]{College of Intelligence Science and Technology, National University of Defense Technology, Changsha 410073, China.} 
    \address[YKZ]{Department of Automation and BNRist, Tsinghua University, Beijing 100084, China.}


	\begin{abstract} 
		Koopman operators have shown significant potential in designing linear model predictive control (MPC) schemes for nonlinear systems on a lifted observable space. Recent advances have tackled the robust Koopman MPC design issue in the presence of modeling errors, relying on the prior estimation of the modeling uncertainty set. However, deriving a robust positively invariant set using a precalculated uncertainty set can be conservative because the uncertainty set bound is time-varying and dependent on the state and control. Additionally, no existing Koopman MPC design has addressed the closed-loop robustness challenge for nonlinear time delayed systems. Thereby, this article presents a robust adaptive Koopman MPC approach with online updates of uncertainty sets for a class of nonlinear time delayed systems. The unknown nonlinear time delayed system is first modeled in a data-driven manner to derive a lifted time delayed Koopman model in the feature space. By analyzing fundamental properties such as controllability and observability, a robust tube-based MPC algorithm is designed for the time delayed Koopman model. The robust adaptive Koopman MPC algorithm with online updates of the uncertainty sets is then presented to reduce conservatism. Closed-loop robustness under exogenous disturbances and asymptotic convergence in the nominal scenario are proven. Finally, numerical examples verify the effectiveness of the proposed approach.
	\end{abstract}
	
\end{frontmatter}

	\section{Introduction}
	Koopman operators, which are typically of infinite dimension, capture the dynamic evolution of state observables, effectively providing a global linearization of nonlinear systems in a lifted observable space~\cite{korda2018linear,strasser2026overview,bevanda2021koopman}. Approximating these Koopman operators with a finite dimension results in a truncated lifted Koopman model, which can serve as a surrogate predictor in the development of model-based controllers, such as model predictive control (MPC)~\cite{peitz2019koopman,otto2021koopman}.\\
	MPC is a model-based, constrained control strategy that reformulates a constrained optimal control problem into a sequence of finite-horizon optimization problems solved online. At each time step, an open-loop control sequence is derived by solving the underlying optimization problem, and only the first control action is implemented. This optimization process is iteratively solved at each subsequent time step following the receding horizon principle.\\
	One of the pioneering works in this field introduced the Koopman operator into MPC as a linear predictor of unknown nonlinear systems~\cite{korda2018linear}. This approach involves lifting the state into a finite observable space through a set of basis functions, with the model parameters derived in a data-driven manner using extended dynamic mode decomposition (EDMD)~\cite{williams2015data,korda2018linear}. However, the resulting Koopman model inevitably introduces non-negligible modeling errors because constructing a finite observable space that is invariant for open-loop forced systems is challenging. Although choosing a considerable number of appropriate basis functions might improve regression precision, this also increases the online computational complexity of MPC.\\
	To address Koopman modeling errors, an offset-free control approach with an external observer was proposed in~\cite{son2020handling}. A previous work presented a robust Koopman MPC design within the framework of tube MPC~\cite{zhang2022robust}, ensuring constraint satisfaction and theoretical guarantees under exogenous disturbances and Koopman modeling errors. 
	Another interesting robust MPC approach using data-driven Koopman operators was developed in~\cite{9867811}. Furthermore, in {\color{black}~\cite{kim2025k}}, an extension to Koopman MPC was addressed for lateral control applications in intelligent vehicles, where stochastic constraints were formulated under the assumption that the uncertainty set is Gaussian distributed.
	It is highlighted that in scenarios under disturbances with unknown distributions, the methods proposed in~\cite{zhang2022robust} and~\cite{9867811} rely on a priori computation of the uncertainty set using estimated Lipschitz constants. However, this precomputed constant uncertainty set can be overly conservative, as it is actually time-varying and dependent on state and control variables. This observation has inspired our robust adaptive Koopman MPC design, which incorporates online updates of uncertainty sets to mitigate conservatism.\\
	Many robotic and biological systems manifest nonlinearities and time delays in their states due to potential delays in material or data transport~\cite{reble2011model,jeong2005constrained}. Among various methodologies for handling time delayed systems, MPC stands out for its capability to manage state and input constraints through online optimization. However, ensuring stability of MPC for nonlinear time delayed systems remains a significant challenge. Previous works, such as \cite{reble2011model} and \cite{raff2007model}, have presented nonlinear MPC approaches for time delayed systems with asymptotic stability guarantees based on the Lyapunov–Krasovskii functional. Additionally, a Lyapunov-based MPC approach was developed for nonlinear systems with delayed measurements in \cite{liu2009lyapunov}. Despite these advancements, robust MPC design for nonlinear time delayed systems under exogenous disturbances remains unresolved.\\
	In this article, we tackle the control of nonlinear time delayed systems through a robust Koopman MPC approach, transforming the nonlinear control problem into a lifted linear control problem. 
Our contributions are summarized as follows. 
	{\color{black}(i) We develop a robust Koopman-based MPC framework for unknown nonlinear systems with time delays, together with rigorous closed-loop guarantees. The key contribution lies in formulating nonlinear time-delay dynamics within a tractable robust tube MPC framework via Koopman lifting, enabling robust constraint satisfaction and recursive feasibility despite model uncertainty and delayed states. To the best of our knowledge, Koopman-based robust MPC for nonlinear time-delay systems with such closed-loop guarantees has not been previously investigated. 
    (ii) We further propose an adaptive Koopman-based MPC scheme that enables online updates of the uncertainty bounds and the associated robust positively invariant (RPI) set. Compared with existing robust Koopman MPC approaches, such as~\cite{zhang2022robust}, which rely on offline computation of fixed RPI sets, the proposed adaptive mechanism reduces conservatism by adjusting the tube size online according to time-varying uncertainty estimated from real state measurements. Moreover, we have proven the recursive feasibility and robustness properties for robust Koopman MPC under online sets updates. } \\
	The rest of the paper is organized as follows. Section~\ref{sec:cpf} introduces the control problem and provides preliminaries on Koopman operators. Section~\ref{sec:design} presents the Koopman model and the robust Koopman MPC design for time delayed systems. In Section~\ref{sec:adaptive}, the adaptive Koopman MPC extension is introduced with online updates of uncertainty sets. Section~\ref{sec:simulation} demonstrates the numerical simulation results. Finally, Section~\ref{sec:con} concludes the paper.\\ 
	\textbf{Notation:} We use $\mathbb{R}$ and $\mathbb{R}^{+}$ to denote the sets of real numbers and positive real numbers, respectively; $\mathbb{R}^{n}$ to denote the $n$-dimensional Euclidean space of real vectors; and $\mathbb{R}^{n \times m}$ to denote the space of $n \times m$ real matrices. The set of positive natural numbers is denoted by $\mathbb{N}$, and $\mathbb{N}_{1}^l$ denotes the set ${1, \ldots, l}$.
Given a variable $r$, we denote the sequence $r_k, \ldots, r_{k+N}$ by ${\bm r}_{\scriptscriptstyle k:k+N}$, and use $\bm r_k$ to represent this sequence after its first appearance, where $k$ is the discrete-time index, and $N$ is a positive integer.
For a matrix $P \in \mathbb{R}^{n \times n}$, $\lambda_{\rm max}(P)$ denotes its largest eigenvalue. For a vector $x \in \mathbb{R}^n$, we use $\|x\|_Q^2$ to denote $x^{\top} Q x$, and $\|x\|$ to denote its Euclidean norm.
Given two sets $\mathcal{Z}$ and $\mathcal{V}$, their Minkowski sum is defined as $\mathcal{Z} \oplus \mathcal{V} = \{z + v \mid z \in \mathcal{Z},\ v \in \mathcal{V} \}$.
For a given set of variables $z_i \in \mathbb{R}^{q_i}$, where $i = 1, 2, \ldots, M$, we define the concatenated vector $(z_1, z_2, \ldots, z_{\scriptscriptstyle M})$ in the compact form: $(z_{1}, z_{2}, \cdots, z_{\rm\scriptscriptstyle M})=[\,z_{1}^{\top}\ z_{2}^{\top}\ \cdots\ z_{\rm\scriptscriptstyle M}^{\top}\,]^{\top}\in{\mathbb{R}}^{q}$, where $q= \sum_{i=1}^{M}q_{i}$.
	\section{Control problem and preliminaries}\label{sec:cpf}
	This section first describes the time delayed systems and formulates the optimal control problem considered, and then reviews the preliminaries on the Koopman operators.
	\subsection{Control problem formulation}
	Consider a class of nonlinear, discrete-time, time delayed systems with additive disturbances described as
	\begin{equation}\label{Eqn:non-model}
		x^{+}=f(x,x_{-\tau},u)+w_o,
	\end{equation}
	with the initial condition on the time delayed state as
	\begin{equation}
		x_{-i}=\varphi_{-i},\, i\in\mathbb{N}_{-\tau}^0,
	\end{equation}
	where $x\in \mathcal{X}\subset \mathbb{R}^{n}$, $u\in \mathcal{U}\subset \mathbb{R}^m$ are the state and control variables, $x_{-\tau}$ is the time delayed state with a delay index $\tau$ which is assumed constant and known, $\mathcal{X}$ and $\mathcal{U}$ are convex \zxl{and compact} sets containing the origin in their interiors, $w_o\in\mathcal{W}_o$ is an additive, unknown but bounded noise, $\mathcal{W}_o$ is a convex and compact set containing the origin in the interior, $f$ is an unknown state transition function. It is assumed that $f(0,0,0)=0$,  $f(x,x_{-\tau},u)$ is $C^{\infty}$ on $\mathcal{X}^2\times \mathcal{U}$, $\|f(x,x_{-\tau},u)\|<+\infty$, 
	{\color{black} and the state $x$ is measured.} Note that while some systems may be directly influenced by multiple delayed states, we consider only a single delayed state $x_{-\tau}$ in the system for simplicity. This simplification does not lead to substantial differences in controller design and theoretical analysis.\\ 
	Given any initial state condition $x_{-i}\in\mathcal{X}$ for $i\in\mathbb{N}_{-\tau}^0$, the control objective is to minimize the performance index $J=\sum_{k=0}^{+\infty} \|x_k\|_{Q}^2+\|u_k\|_R^2,$
	where $Q=Q^{\top}\in\mathbb{R}^{n\times n}$ and $R=R^{\top}\in\mathbb{R}^{m\times m}$, $Q,R\succ 0$. 
    %
	\begin{definition}[Local controllability~\cite{4484215}]\label{def-control}
		The time delayed system $x^{+}=f(x,x_{-\tau},u)$ is controllable in the domain $\mathcal{X}^2\times \mathcal{U}$ if, for any $x_{-i}\in\mathcal{X}$, $i\in\mathbb{N}_0^{\tau}$, there exists a feedback control law $u(x_k,x_{k-\tau})\in\mathcal{U}$, $u(0,0)=0$, such that the state reaches the origin in finite time, i.e., there exists a finite time instant $l$ for which $x_l=0$. 
	\end{definition}
	    \begin{remark}\label{remark1} Designing optimization-based control for the time delayed system~\eqref{Eqn:non-model} under unknown dynamics is  nontrivial. In principle, if the system dynamics were known, one could construct an augmented model by incorporating the delayed states, transforming the problem into a delay-free form. 
        This would allow using standard nonlinear MPC techniques. {\color{black}However, the augmented state space grows to a dimension of $n\times(\tau+1)$, leading to a substantial increase on online computational complexity. To enable efficient online computation and handle unknown dynamics, we adopt a robust MPC framework based on a delay-aware Koopman modeling approach, as detailed in the following sections.}
    \end{remark}
	\vspace{-1mm}
	\subsection{Preliminary on Koopman operators}
    \vspace{-1mm}
	{\color{black}
		We first review the Koopman operator for a class of discrete-time nonlinear delay-free systems $x^{+}=f(x,u)$.
		Let $\vec{\bm u}=(u_{i})_{i=1}^{+\infty}$  be the collection of all the control inputs belonging to $\mathcal{U}^{\infty}$ and $\Gamma$ be a left shift operator such that {\color{black}$\vec{\bm u}(i+1)=\Gamma \vec{\bm u}(i)$, where $\vec{\bm u}(i)$ denotes the $i$-th element of $\vec{\bm u}$.} Define $\phi(x,\vec{\bm u})$ be a scalar-value observable function with arguments in $x,\,\vec{\bm u}$, and denote {\color{black}$\mathcal{F}\in \mathbb{C}$} as a given space in which the observables lie. The Koopman operator $\mathcal{K}:\mathcal{F}\rightarrow \mathcal{F}$ is defined as~\cite{arbabi2018data,korda2018convergence}
  {\color{black}\begin{equation}\label{Eqn:koopman}
			\mathcal{K} \phi(x,\vec{\bm u})=\phi (f(x,u), \Gamma\vec{\bm u}),
		\end{equation}
		for every $\phi(x,\vec{\bm u})\in\mathcal{F}$, where $\mathcal{F}$ is invariant under the action of the Koopman operator.} 
		%
		\section{Robust Koopman MPC for time delayed systems}\label{sec:design} 
		In this section, we first derive the Koopman model for time delayed systems. Then, we present a robust Koopman MPC (r-KMPC) approach to address the time delays in the Koopman model. Finally, the closed-loop robustness and point-wise convergence are proven.
		\subsection{Koopman model for time delayed systems}
		%
        To formulate the Koopman operator for time delayed systems, we first define the augmented state $\bm x_d=(x_i)_{i=-\tau}^0$ as an augmented state with time delayed states. As a result, the time delayed model~\eqref{Eqn:non-model} can be transformed into a higher dimensional standard dynamical model: $\bm x_d^+=F(\bm x_d,u_w)$ {\color{black} with $u_w=(u,w_o)$}. Let $\vec{\bm u}_w=(u_{w,i})_{i=1}^{+\infty}$, then the Koopman operator for~\eqref{Eqn:non-model} is given as
		{\color{black}\begin{equation}\label{Eqn:koopman-delay}
			\mathcal{K} \phi_d(\bm x_d,\vec{\bm u}_w)=\phi_d(F(\bm x_d, u_w), \Gamma\vec{\bm u}_w), 
		\end{equation}
		for every $\phi_d(\bm x_d,\vec{\bm u}_w)\in\mathcal{F}_d$, where $\mathcal{F}_d$ is the invariant observable space of $\phi_d(\bm x_d,\vec{\bm u}_w)$.} 
     \begin{remark}  In principle, the EDMD method~\cite{korda2018linear} for delay-free systems can be used for the finite-dimensional approximation of $\mathcal{K}$ for model-based controller design. This is achieved by selecting a computable observable function of the form $\Phi_d(\bm{x}_d, u_w) = \left( \Psi(\bm{x}_d), u, w_o \right)$. {\color{black}However, as noted in Remark~\ref{remark1}, this choice could increase the dimensionality of the Koopman state space, as $\bm{x}_d$ spans $\tau+1$ time steps. To reduce the complexity of the lifted space, we define a more practical observable function using only the current and $\tau$-step delayed states, $x$ and $x_{-\tau}$, which directly affects the state evolution.} A vector of observable functions is given as
		\begin{equation}\label{Eqn:lift-extend-time delay}
			\Phi_d(x,x_{-\tau},u_w)=\left(
			\Psi(x),\Psi(x_{-\tau}),
			u,w_o\right)\in\mathbb{R}^{N_{\Phi}},
		\end{equation}
		where $N_{\Phi}=2n_{\psi}+m+n$, $\Psi(\cdot)$ is defined as \begin{equation}\label{eqn:psi_x}
			\Psi\left(x\right) :=( {\psi_{1}\left(x\right)},{\cdots}, {\psi_{n_{\psi}}\left(x\right)}),
		\end{equation}
		and $\psi_i$, $i\in\mathbb{N}_{1}^{n_{\psi}}$, can be chosen as some basis functions or neural networks~\cite{lian2021koopman}.
        \end{remark}
 Let $\mathcal{K}_{\scriptscriptstyle N_{\Phi}}\in\mathbb{R}^{N_{\Phi}\times N_{\Phi}}$ be a finite-dimensional approximation of $\mathcal{K}$. The goal is to derive $\mathcal{K}_{\scriptscriptstyle N_{\Phi}}$ via minimizing $
\|\Phi_d(z^+)-\mathcal{K}_{\scriptscriptstyle N_{\Phi}}\Phi_d(z)\|^2,$ where $z=(x,x_{-\tau},u,\hat w_o)$, $\hat w_{o}\in\hat{\mathcal{W}}_o$ is an estimation of the unknown disturbance $w_{o}$, wherein $\hat w_{o}$ can be chosen such that $\hat{\mathcal{W}}_o$ is compact. 
  {\color{black} To guarantee the approximation accuracy of the Koopman operator, it is assumed that the data points $\{(x_i^+,x_i,x_{i-\tau},u_i,\hat w_{o,i})\}_{i=1}^M$ are drawn independently according to a given probability distribution $\mu$. Alternatively, this condition can be replaced by assuming that the generated data points $\{(x_i^+,x_i,x_{i-\tau},u_i,\hat w_{o,i})\}_{i=1}^M$ are ergodic in $\mathcal{X}^2\times\mathcal{U}$ with respect to the distribution $\mu$~\cite{korda2018convergence}.}
  
  Let the first $n_{\psi}$ rows of $\mathcal{K}_{\scriptscriptstyle N_{\Phi}}$ be $[\mathcal{K}_{\scriptscriptstyle N_{\Phi}}]_{\scriptscriptstyle 1:n_{\psi}}=[A\, A_d\, B\, D]\in\mathbb{R}^{n_{\psi}\times N_{\Phi}}$, where $A,$ $A_d\in\mathbb{R}^{n_{\psi}\times n_{\psi}}$, $B\in\mathbb{R}^{n_{\psi}\times m}$, and $D\in\mathbb{R}^{n_{\psi}\times n}$ are derived directly from data.
		Since we only care about the prediction $\Psi(x^+)$ with $\Phi_d(z)$ for controller design, the following least square problem is stated~\cite{korda2018linear}: 
\begin{subequations}\label{Eqn:appro_K-delay}
			\begin{align}
				\hspace{-1mm}	\min{
						\scriptstyle [\mathcal{K}_{\scriptscriptstyle N_{\Phi}}]_{\scriptscriptstyle 1:n_{\psi}}} 
            \hspace{-1mm}\sum_{i=1}^{M}&\|
                \Psi(x_i^+)-[\mathcal{K}_{\scriptscriptstyle N_{\Phi}}]_{\scriptscriptstyle 1:n_{\psi}}\Phi_d(z)\|^2,	
			\end{align}
    The state $x$ is recovered from $\Psi(x)$ through a matrix $C\in\mathbb{R}^{n\times n_{\psi}}$ via the following least square problem:
            \begin{align}
            \min{C}\sum_{i=1}^{M}&\left\|C\Psi(x_i)-x_i\right\|^2.
            \end{align}
\end{subequations}
Note that, the practical observable function~\eqref{Eqn:lift-extend-time delay} and the data-driven problem~\eqref{Eqn:appro_K-delay} inevitably introduce modeling uncertainties~\cite{zhang2022robust}. Letting $s=\Psi(x)$ and $s_{-\tau}=\Psi(x_{-\tau})$, 
an equivalent Koopman model of~\eqref{Eqn:non-model} with time delays and uncertainties can be written as
		\begin{equation}\label{Eqn:linear_p-residual}
			\left\{\hspace{-1mm}\begin{array}{l}
				{s}^+=A{s}+A_{d}s_{-\tau}+Bu+\bar w\\
				[0.2cm] x=C{s}+v(s),
			\end{array}\right. \qquad
		\end{equation}
		where $\bar w=D\hat w_o+w(s,s_{-\tau},u,w_o,\hat w_o)\in\bar {\mathcal{W}}$,  $v(s)\in \mathcal{V}$ and $w(s,s_{-\tau},u,w_o,\hat w_o)\in\mathcal{W}$ are the modeling errors, wherein $\mathcal{W}$ and $\mathcal{V}$ are convex sets containing the origin; $\bar {\mathcal{W}}=D \hat{\mathcal{W}}_o\oplus \mathcal{W}$, $\hat{\mathcal{W}}_o$ is a computable convex set where $\hat w_{o}$ lies in. It is assumed that $\hat{\mathcal{W}}_o$ is bounded.} 
        The boundedness of $\bar{\mathcal{W}}$ and $\mathcal{V}$ will be formally established in Lemma~\ref{lemma:sets}.
        \vspace{-1mm}
	{\color{black}	\begin{assumption}[Lifted function~\cite{zhang2022robust}]\label{assum:invers-phi}\hfill
	\begin{enumerate}[(i)]
    \item The lifted function $\Psi(x)$ is Lipschitz continuous;
    \item $\{\psi_i(x)\}_{i=1}^{n_{\psi}}$ are linearly independent;
    \item\label{ASSUMP:3} ${\rm{rank}}(\triangledown_x \Psi(x))=n,$ $\forall x\in\mathcal{X}$.
    \end{enumerate}
		\end{assumption}}
        \vspace{-1mm}
		Let the set $\mathcal{S}_{\Psi}=\{s\in\mathbb{R}^{n_{\psi}}|s=\Psi(x), x\in\mathcal{X}\}$. {\color{black}As shown in~\cite{zhang2022robust}, under Assumption~\ref{assum:invers-phi}.(\ref{ASSUMP:3}), there  exists a locally Lipschitz function $\Psi^{-1}:\mathcal{S}_{\Psi}\rightarrow\mathcal{X}$ such that $\Psi^{-1}(\Psi(x))=x$, $\forall\, x\in\mathcal{X}$. In view of this property, the lifting function $\mathcal{X}$ can then be chosen from a broad class of commonly used observables (e.g., polynomials, radial basis functions), without requiring restrictive constructions such as bounded mappings over unbounded domains.}

		\begin{lemma}[Sets and equilibrium~\cite{zhang2022robust}]\label{lemma:sets}\hfill
			\begin{enumerate}[(i)]
				\item 	If $\Psi(x)$, $x\in\mathcal{X}$, is such that $\mathcal{S}_{\Psi}$ is bounded, then $\bar{\mathcal{W}}$ and $\mathcal{V}$ are bounded.
				\item Let $\bar A=\begin{bmatrix}
					I-(A+A_d)^{\top}&C^{\top}
				\end{bmatrix}^{\top}$ and $(u,s,s_{-\tau},x)=(0,s_r,s_r,0)$ be the equilibrium of model~\eqref{Eqn:linear_p-residual} under $w_o=0$. It holds that $\left(w(s_r,s_r, 0,0,0),v(s_r)\right)=0$ iff
				\begin{equation}\label{Eqn:equili-uncer=0}
					s_r\subseteq {\rm{Ker}}\,\bar A \quad \text{or}\quad s_r=0; 
				\end{equation}
				moreover, if the pair $(A+A_d,C)$ is observable, condition~\eqref{Eqn:equili-uncer=0} is reduced to $s_r=0.$ 
					
			\end{enumerate}
		\end{lemma}
		\textbf{Proof.} Claim (i): The proof follows from~\cite{zhang2022robust} but with several changes due to the time delays.  First, in view of $s=\Psi(x)$, one has  $w=\Psi(f(\Psi^{-1}(s),\Psi^{-1}(s_{-\tau}),u)+w_o)-As-A_ds_{-\tau}-Bu-D\hat w_o$. Letting $\delta w=w_o-\hat w_o$, for all $s,\,s_{-\tau}\in\mathcal{S}_{\Psi}$, %
		$u\in\mathcal{U}$, $\hat w_{o}\in\hat{\mathcal{W}}_o$, and $w_{o}\in\mathcal{W}_o$, it is convenient to write the following inequality:
		\vspace{-1mm}\begin{equation}\label{Eqn:lipsch-1}
			\begin{array}{lll}
				\|w\|\leq\\
				\leq L_s \|s\|+L_{s,d} \|s_{-\tau}\|+
				L_u \|u\|+L_{\delta w} \|\delta w_{o}\|+L_{\hat w}\|\hat w_{o}\| \\
				<+\infty,\vspace{-1mm}
			\end{array}
		\end{equation}
		where $L_s$, $L_{s,d}$, $L_u$, $L_{\delta w}$, and $L_{\hat w}$ are bounded Lipschitz constants. {\color{black}Hence,  $\bar{\mathcal{W}}$ is bounded since $\mathcal{X}$, $\mathcal{U}$, $\hat{\mathcal{W}}_o$ and $\mathcal{W}_o$ are compact.} Analogously,  $\mathcal{V}=\mathcal{X}\ominus C\mathcal{S}_{\Psi}$ is bounded. \\
		Claim (ii): Condition~\eqref{Eqn:equili-uncer=0} follows straightly from linear algebra. Also, provided that $(A+A_d,C)$ is observable,~\eqref{Eqn:equili-uncer=0} is equivalent to
		\begin{equation}\label{Eqn:linear-obser}
			\begin{array}{l}
				Cs_r=C(A+A_d)s_r=\cdots=
				C(A+A_d)^{n_{\psi}-1}s_r=0,
			\end{array}
		\end{equation}
		since $s_r=(A+A_d)s_r=(A+A_d)^2s_r\cdots$ from  $s_r=(A+A_d)s_r$,
		leading to the unique solution $s_r=0$.
		\hfill $\blacksquare$\\
		The following basic assumption is given in line with~\cite{zhang2022robust}.
		\begin{assumption}[\cite{zhang2022robust}]\label{assu:phi-ori}
			The lifted function $\Psi(x)$ is constructed such that  $\Psi(0)= 0$.
		\end{assumption}
        \begin{remark}
Under Assumption~\ref{assu:phi-ori}, the equilibrium of system~\eqref{Eqn:linear_p-residual} with 
$w_o=0$ is located at the origin, as shown in Lemma~\ref{lemma:sets}. Moreover, consistent with~\cite{zhang2022robust}, it follows that model~\eqref{Eqn:linear_p-residual} is controllable on the set $\mathcal{S}_{\Psi}^2\times \mathcal{U}$ under 
$w_o=0$ if and only if the original system~\eqref{Eqn:non-model} is controllable on the set $\mathcal{X}^2\times \mathcal{U}$. This implies the existence of a control law that drives model~\eqref{Eqn:linear_p-residual} to the origin when 
$w_o=0$, which is essential for ensuring closed-loop stability.
Furthermore, Lemma~\ref{lemma:sets} establishes that if the lifted state of model~\eqref{Eqn:linear_p-residual}
$s_k\rightarrow 0$ as $k\rightarrow +\infty$, then the state of the original model~\eqref{Eqn:non-model} satisfies 
$x_k\rightarrow 0$ as $k\rightarrow +\infty$. This connection reinforces the validity of the stability analysis with the lifted Koopman model.
        \end{remark}
        The following lemma is introduced for the stability analysis of time delayed systems.
		{\color{black}\begin{lemma}[Lemma 6.1 in~\cite{Fridman2014TimeDelay}]
			System \eqref{Eqn:linear_p-residual} with $w_o=0$ under a feedback contol law $u=h(s,s_{-\tau})$ is asymptotically stable if there exist positive numbers $\alpha$, $\beta$, and a Lyapunov-Krasovskii functional $V_{\rm LK}:\mathbb{R}^{n_{\psi}(\tau+1)}\rightarrow \mathbb{R}^{+}$ such that
			\begin{subequations}\label{EQN:LY-K}
				\begin{align}
					&0\leq V_{\scriptscriptstyle{\rm LK}}(s_k)\leq \beta\left\{\max_{j\in [-\tau,0]}\|s_{k+j}\|^2\right\}\label{eqn:LK-1}\\
					&V_{\scriptscriptstyle{\rm LK}}(s_{k+1})-V_{\scriptscriptstyle{\rm LK}}(s_k)\leq -\alpha \|s_k\|^2.\label{eqn:LK-2}
				\end{align}
			\end{subequations}
		\end{lemma}}
		
	\subsection{Robust Koopman MPC for time delayed systems}\label{sec:r-KMPC}
	Define the nominal time delayed predictor of~\eqref{Eqn:linear_p-residual} as
	\begin{equation}\label{Eqn:unpert}
		\left\{\hspace{-1mm}\begin{array}{l}
			\hat{s}^+=A\hat{s}+A_{d}\hat s_{-\tau}+B\hat u\\
			[0.2cm] \hat x=C\hat {s},
		\end{array}\right. \qquad
	\end{equation}
	where $\hat u$, $\hat s$, and $\hat s_{-\tau}$ are decision variables computed with a standard MPC (deferred in~\eqref{Eqn:optimiz1}). 
	\begin{definition}[Observability]\label{def-observa}
		Model~\eqref{Eqn:unpert} is observable if there exists a finite time instant such that the initial state $s_0$ can be uniquely determined from $u_k,\,\hat x_k,\{\hat x_{k-j}\}_{j=1}^{\tau}$. 
	\end{definition}
	Denote a matrix $G_k$ for $k\in\mathbb{N}$ as
	\begin{equation}
		G_k=\left\{\begin{array}{ll}
			A^k,\, &k=0,\cdots,\tau\\
			G_{k-1}A+G_{k-\tau-1}A_d,\, &k=\tau+1,\cdots.
		\end{array}\right.
	\end{equation}
	Define a matrix $L_k$ for $k\in\mathbb{N}$ as
	\begin{equation}
		L_k=\left\{\begin{array}{ll}
			A^k,\, &k=0,\cdots,\tau\\
			\sum_{j=1}^aA^{b+1-j}A_d^{j-1}L_{k-j}, \, &k=\tau+1,\cdots,
		\end{array}\right.
	\end{equation}
	where $a=\lfloor k/\tau\rfloor$, $b=k\bmod \tau$.
	The following proposition regarding controllability and observability is stated.
	\begin{proposition}[Controllability and observability]\label{pro:sta-obser}\hfil
		\begin{enumerate}[(i)]
			\item\label{prop:contr} Model~\eqref{Eqn:unpert} is controllable iff there exists $l\in\mathbb{N}$ such that \begin{equation}\label{Eqn:con-condition}
				\text{rank}\{G_0,\,G_1B,\,\cdots,\,G_lB\}=n_{\psi};
			\end{equation}
			\item Model~\eqref{Eqn:unpert} is observable iff there exists $l\in\mathbb{N}$ such that
			\begin{equation}\label{Eqn:obs-condition}
				\text{rank}\{L_0,\,CL_1,\,\cdots,\,CL_l\}=n_{\psi}.
			\end{equation}
		\end{enumerate}
	\end{proposition}
	\textbf{Proof.} Claim (i): The forward evolution of $\hat x$ and $\hat s$ with~\eqref{Eqn:unpert} at a generic time instant $l$ can be written as 
	\begin{equation}
		\left\{\begin{array}{ll}
			\hat s_l&=\chi_l(\hat s_{-\tau},\cdots,\hat s_{-1})+L_l\hat s_0+\sum_{i=0}^{l}G_{l-i}B\hat u_i\\
			\hat x_l&=C\hat s_l.
		\end{array}
		\right.
	\end{equation}
	In view of the definition of controllability  in Definition~\ref{def-control}, one can conclude that model~\eqref{Eqn:unpert} is controllable iff there exist $l\in\mathbb{N}$ and controls $\hat u_0,\,\cdots,\hat u_l$ such that $\hat s_l=0$ from finite initial states $\{\hat s_{-j}\}_{j=0}^{\tau}$, which leads to condition~\eqref{Eqn:con-condition}. \\ Claim (ii): Analogously, in view of Definition~\ref{def-observa}, a sufficient and necessary condition to uniquely determine $\hat s_0$ with $\hat u,\{\hat x_{-j}\}_{j=0}^{\tau}$ is $\text{rank}\{L_0,\,CL_1,\,\cdots,\,CL_l\}=n_{\psi}$, which is condition~\eqref{Eqn:obs-condition}.
	\hfill $\blacksquare$\\ 
	Let $e={s}-\hat {s}$ and $e_{\scriptscriptstyle {-\tau}}={s_{-\tau}}-\hat {s}_{-\tau}$.	The control action of r-KMPC for time delayed systems is given as
	\begin{equation}\label{Eqn:real_u}
		u=\hat u+Ke+K_de_{\scriptscriptstyle {-\tau}},
	\end{equation}
	{\color{black}where $K\in\mathbb{R}^{m\times n_{\psi}}$ and $K_d\in\mathbb{R}^{m\times n_{\psi}}$ are gain matrices such that $\bar F=F+F_d$ is Schur stable, wherein $F=A+BK$, $F_d=A_d+BK_d$. 
    Note that $\bar F=F+F_d$ being Schur stable may not be guaranteed by Proposition~\ref{pro:sta-obser}.(\ref{prop:contr}). Hence, we additionally require the following assumption to hold.
	\begin{assumption}[Stabilizability]\label{assum:stabi-obser}
		The pair $(A+A_d,B)$ is stabilizable.
	\end{assumption}}
{\color{black}Note that, Assumption~\ref{assum:stabi-obser} imposes an additional condition beyond the controllability property stated in Proposition~\ref{pro:sta-obser} (see also Definition~\ref{def-control}). This additional condition is beneficial for the construction of an RPI set for the error state $e$, as detailed below in~\eqref{Eqn:e}.} \\
    From~\eqref{Eqn:linear_p-residual} and~\eqref{Eqn:real_u}, the error $e$ evolves according to the following unforced time delayed system:
	\begin{equation}\label{Eqn:e}
		\left\{\begin{array}{l}
			e^+=Fe+F_de_{\scriptscriptstyle -\tau}+\bar w\\
			[0.2cm] e_x=C e+v,
		\end{array}\right. \qquad
	\end{equation}
	where $e_x=x-\hat x$.	Let $\mathcal{Z}_s$ be an RPI set of $e_s$ such that for any $e,\,e_{-\tau}\in\mathcal{Z}_s$, one has $\mathcal{Z}_s\subseteq \bar F\mathcal{Z}_s\oplus \bar{\mathcal{W}}$. {\color{black} Under this condition, a minimal RPI set (if possible) of $\mathcal{Z}_s$ can be constructed as $\mathcal{Z}_s=\bigoplus_{i=0}^{\infty}\left(F+F_d\right)^i \mathcal{W}$, which is nonempty, convex, and compact.} Consequently, we have $e_x\in C\mathcal{Z}_s\oplus \mathcal{V}:=\mathcal{Z}_x$.\\
	Now we are ready to state the nominal MPC problem to compute $\hat u$, $\hat s$, and $\hat s_{-\tau}$ in~\eqref{Eqn:real_u}.
	At any time instant $k$, the following quadratic optimization problem is to be solved:
	\begin{equation}\label{Eqn:optimiz1}
		\min{\{\hat s_{k-j}\}_{j=0}^{\tau},{\bm {\hat u}}_{k}}  {V\big(\hat{s}_k,\bm {\hat u}_{\scriptscriptstyle k}\big)}
	\end{equation}
	where
	\begin{equation}\label{Eqn:V_b}
		\begin{array}{ll}
			V\big(\hat{s}_k,\bm {\hat u}_{\scriptscriptstyle k}\big)=
			\sum_{i=0}^{N-1}(\|\hat {x}_{k+i}\|_{ Q}^2+\|\hat u_{k+i}\|_{R}^2)
			+V_f(\hat{s}_{k+N}),
		\end{array}
	\end{equation}
	and $\{\hat s_{k-j}\}_{j=0}^{\tau}$, and ${\bm {\hat u}}_{\scriptscriptstyle k}$ are the decision variables, $N$ is the prediction horizon, the terminal cost $V_f(\hat{s})$ is defined as $V_f(\hat{s})=\|\hat{s}\|^2_P+\sum_{i=1}^{\tau}\|\hat{s}_{-i}\|_{S}^2,$
	where the symmetric positive-definite matrices $P$ and $S$ satisfy the following inequality:
	\begin{equation}\label{Eqn:LYA}
		\Pi(P,S)=\begin{bmatrix}\Pi_1(P,S)&\star\\
			F_d^{\top}PF+K_d^{\top}RK&\Pi_2(P,S)
		\end{bmatrix}\preceq 0,
	\end{equation}
	where $\Pi_1(P,S)=F^{\top}PF+S-P+\bar Q+K^{\top}RK$, $\Pi_2(P,S)=F_d^{\top}PF_d-S+K_d^{\top}RK_d$,  $\bar Q=C^{\top}QC$.
	
	The optimization problem~\eqref{Eqn:optimiz1} is performed subject to the following constraints:
	\begin{enumerate}[1)]
		\item The nominal time delayed model~\eqref{Eqn:unpert};
		\item Tighter state and control constraints:
		\begin{subequations}\label{Eqn:constraint-MPC}\vspace{-2mm}
			\begin{align}
				\hat {s}_{k+i},\hat {s}_{k+i-\tau}\in   {\mathcal{S}},\,\,i=0,\dots,N-1\\
				\hat u_{k+i}\in \hat{ {\mathcal{U}}},\,\,i=0,\dots,N-1, \vspace{-2mm}
			\end{align}
			where ${\mathcal{S}}=\{\hat{s}|C\hat s\in  {\mathcal{X}}\ominus{\mathcal{Z}}_x\}$, ${\hat {\mathcal{U}}}={\mathcal{U}}\ominus (K+K_d){\mathcal{Z}}_s$;
			\item The initial and terminal state constraints \vspace{-2mm}
			\begin{align}
				e_{k-j}\in \mathcal{Z}_s,\,& j\in\mathbb{N}_0^{\tau}\label{Eqn:inin-condi}\\
				\hat {s}_{k+N-j}\in {{\mathcal{S}}_f},\, &j\in\mathbb{N}_0^{\tau},\vspace{-2mm}
			\end{align}
		\end{subequations}
	\end{enumerate}
	where ${\mathcal{S}}_f$ is a positive invariant set of~\eqref{Eqn:unpert} under state and control constraints such that $\bar F{\mathcal{S}}_f\subseteq {\mathcal{S}}_f$. 
	
    {\color{black}Let $\mathbb{S}_N$ be the set of lifted states in $\mathcal{S}$ for which the optimization problem~\eqref{Eqn:optimiz1} has a solution, i.e., 
    $\mathbb{S}_N:=\{\{\hat s_{k-j}\}_{j=0}^{\tau}\in\mathcal{S}^{\tau+1}|\mathbb{U}_N\neq \varnothing\},$
    where $\mathbb{U}_N:=\{{\bm u}_0|\text{constraint}\,\eqref{Eqn:constraint-MPC}\, \text{is  fulfilled}\}$.} The following conditions are assumed holding. 
    {\color{black}\begin{assumption}\label{assum:recur-feasi}\hfill
    \begin{enumerate}
    \item The conditions $\{\bm 0\}\subset\hat{\mathcal{U}}$ and $\{\bm 0\}\subset{\mathcal{S}}$ hold;
    \item The optimization problem~\eqref{Eqn:optimiz1} is feasible at the initial time instant $k=0$, i.e., $\mathbb{S}_{N}\neq \varnothing$.
    \end{enumerate}
    \end{assumption}}
	\begin{theorem}[Closed-loop properties]\label{the:rmpc-robust} {\color{black}Under Assumptions \ref{assum:invers-phi}-\ref{assum:recur-feasi}, given $s_{-j}\in\mathbb{S}_N\oplus \mathcal{Z}_s$, $\forall \, j\in\mathbb{N}_0^{\tau}$,} it holds that:
		\begin{enumerate}[(i)] 
			\item {\color{black}the optimization problem~\eqref{Eqn:optimiz1} is recursively feasible at all times $k\in\mathbb{N}$;}
			\item 	the lifted nominal system~\eqref{Eqn:unpert} using $\hat s_k,\hat s_{k-\tau},\hat u_k$ computed with~\eqref{Eqn:optimiz1} converges to the origin asymptotically {\color{black}with a region of attraction $\mathbb{S}_N$}, i.e., 	$\hat s_k,\, \hat x_k,\,\text{and}\,\hat u_k\rightarrow 0\ \text{as}\ k\rightarrow+\infty;$
			\item the state $s$ and the control $u$ of the closed-loop system~\eqref{Eqn:linear_p-residual} with~\eqref{Eqn:real_u}  are such that
			$s_k\rightarrow{\mathcal{Z}}_s\ \text{and}\ u_k\rightarrow K{\mathcal{Z}}_s\ \text{as}\ k\rightarrow+\infty;$
			and $x_k\rightarrow {\mathcal{Z}}_x\ \text{as}\ k\rightarrow+\infty.$
		\end{enumerate}		
		\end{theorem}
		\textbf{Proof.} Claim (i): {\color{black}We prove this claim by mathematical induction.} As a first step, we assume that at time $k$, the optimal decision variables $\{\hat s_{k-j|k}\}_{j=0}^{\tau}$, $\bm {\hat u}^{\ast}_{k|k}$ are calculated with~\eqref{Eqn:optimiz1} (associated with the optimal cost $V^{\ast}_k$),
		such that $\hat{s}_{k+i},\hat{s}_{k+i-\tau}\in {\mathcal{S}}$,  $\hat {u}_{k+i}\in{\hat{\mathcal{U}}}$, $\forall i\in\mathbb{N}_0^{N-1}$, $\hat {s}_{k+N-j|k}\in{\mathcal{S}}_f$, $\forall j\in\mathbb{N}_0^{\tau}$, as well as the real state and control constraints being verified, i.e., $x_{t-j|k},x_{t|k}\in\mathcal{X}$ and $u_{t|k}\in\mathcal{U}$ for all $t\geq k$ and $j\in\mathbb{N}_0^{\tau}$. At the next time $k+1$, set $\hat{s}_{k+1-j|k+1}=\hat{s}_{k+1-j|k}$ for $j\in\mathbb{N}_0^{\tau}$, $\bm{\hat u}^s_{k+1}=\bm {\hat u}^{\ast}_{k+1:k+N-1|k}, K\hat{s}_{k+N|k}+K_d\hat{s}_{k+N-\tau|k}$ (associated with a sub-optimal cost $V^{s}_{k+1}$), {\color{black}which basically shifts previously optimal sequence forward by one step, leading to a feasible but suboptimal solution}. Then, one has  ${s}_{k+1-j}-\hat{s}_{k+1-j|k+1}\in {\mathcal{Z}}_s,$ for $j\in\mathbb{N}_0^{\tau}$, and $ 
		\hat s_{k+i+1|k+1},\hat s_{k+i+1-\tau|k+1}\in {\mathcal{S}},\, \forall i\in\mathbb{N}_{1}^{N-1}$ in view of the definition of ${\mathcal{S}}$. Also, one has $\hat{s}_{k+N+1-j|k+1}\in{\mathcal{S}}_f$, $\forall j\in\mathbb{N}_0^{\tau}$, in view of the definition of ${\mathcal{S}}_f$.  
		Hence,~\eqref{Eqn:optimiz1} is feasible at time $k+1$. The recursive feasibility of r-KMPC holds. \\
		Claim (ii): Note that $V^{\ast}_{k+1}\leq V^s_{k+1}$,   it follows that:
\begin{equation}\label{Eqn:V-MONO-r}
			\begin{array}{ll}
				&V^{\ast}_{k+1}- V^{\ast}_k\leq V^{s}_{k+1}- V^{\ast}_k\\
				&=-\|\hat {x}_k\|_{Q}^2-\|\hat u_k\|_{R}^2+\\
				&\underbrace{\left\{V_f(\hat {s}_{k+N+1})-V_f(\hat {s}_{k+N})\right\}}+\|\hat {s}_{k+N}\|_{\bar Q}^2+\|\hat u_{k+N}\|_{R}^2.\\
				&\hspace{17mm}\Delta V_f
			\end{array}
		\end{equation}
		Note  that $V_f(\hat s_{k+N})\geq 0$ and $V_f(\hat s_{k+N})\leq \|\hat s_{k+N}\|_P^2+\tau \left\{\max_{j\in [-\tau,0]}\|\hat s_{k+N+j}\|_S^2\right\}\leq (\lambda_{\rm max}(P)+\tau\lambda_{\rm max}(S))\cdot\break\left\{\max_{j\in [-\tau,0]}\|\hat s_{k+N+j}\|^2\right\}$, which satisfies~\eqref{eqn:LK-1}. Also, in view of the definition of $\Pi(P,S)$ and under the feedback control law $u_{k+N}=K\hat{s}_{k+N}+K_d\hat{s}_{k+N-\tau}$, one has $$\Delta V_f+\|\hat {s}_{k+N}\|_{\bar Q}^2+\|\hat u_{k+N}\|_{R}^2=\|(\hat {s}_{k+N},\hat {s}_{k+N-\tau})\|_{\Pi(P,S)}^2.$$
		where the equality is due to~\eqref{Eqn:LYA}. Hence, $\Delta V_f\leq -\|\hat {s}_{k+N}\|_{\bar Q}^2-\|\hat u_{k+N}\|_{R}^2$, leading to~\eqref{eqn:LK-2} being fulfilled. In other words, $V_f(\hat{s})$ is a Lyapunov-Krasovskii functional satisfying~\eqref{EQN:LY-K}. 
		Promptly, it holds that $V^{\ast}$ is bounded since $V_f$ is bounded, and from~\eqref{Eqn:V-MONO-r} one derives that $$V^{\ast}_{k+1}- V^{\ast}_k\leq -\|\hat {x}_k\|_{Q}^2-\|\hat u_k\|_{R}^2.$$
		Then, from~\eqref{Eqn:V-MONO-r},  $V^{\ast}_{k+1}- V^{\ast}_k\rightarrow0$ as $k\rightarrow+\infty$. Consider also $ V^{\ast}_k-V^{\ast}_{k+1}\geq \|\hat {x}_k\|_{ Q}^2+\|\hat u_k\|_{R}^2,$ hence $\hat u_k\rightarrow 0$, $\hat x_k\rightarrow 0$, as $k\rightarrow +\infty$, in view of the positive-definiteness of $R$ and $Q$. Also, $\hat{s}_k\rightarrow 0$ as $k\rightarrow +\infty$ due to $\Psi(0)=0$ in view of Assumption~\ref{assu:phi-ori}.\\
		Claim (iii):
		Recall that $x\in\hat x\oplus \mathcal{Z}_x$, $u\in\hat u\oplus (K+K_d)\mathcal{Z}_s$, and $s\in\hat s\oplus \mathcal{Z}_s$, it holds that, $ x_k\rightarrow {\mathcal{Z}}_x\ \text{ and }\ u_k\rightarrow (K+K_d){\mathcal{Z}}_s\ \text{ as }\ k\rightarrow +\infty,$ and $s_k\rightarrow{\mathcal{Z}}_s\ \text{as}\ k\rightarrow +\infty.$
		\hfill$\blacksquare$
		We next show the asymptotic stability under nonexogenous disturbance, i.e. $w_0=0$. In view of the convergence of the nominal system~\eqref{Eqn:unpert}, one can write from~\eqref{Eqn:lipsch-1} that
		\begin{equation}
			\begin{array}{lll}
				\|\bar w\|&\leq& L_s \|s\|+L_{s,d} \|s_{-\tau}\|+L_u \|u\|\\
				&\leq&\|s\|_{\bar L_s}+\|s_{-\tau}\|_{\bar L_{s,d}},
			\end{array}
		\end{equation} 
		where $\bar L_s=(L_s+L_uK^{\top}K)$ and $\bar L_{s,d}=(L_{s,d}+L_uK_d^{\top}K_d)$.
		\begin{theorem}[Asymptotic stability]\label{the:rmpc-convergence} If $w_o=0$ and there exists a scalar $\gamma>0$ such that \begin{equation}\label{Eqn:iss-con}
				\begin{array}{ll}
					\tilde{\Pi}(P,S)\hspace{-2mm}&=\begin{bmatrix}
						{\Pi}_1(P,S)+\gamma\bar L_s&\star&\star\\F_d^{\top}PF+K_d^{\top}RK&{\Pi}_2(P,S)+\gamma\bar L_{s,d}&\star\\
						PF&PF_d&P-\gamma I
					\end{bmatrix}\vspace{2mm}\\
					\hspace{-2mm}&\preceq 0,
				\end{array}
			\end{equation}
			the  closed-loop system~\eqref{Eqn:linear_p-residual} and~\eqref{Eqn:non-model} with~\eqref{Eqn:real_u} converge to the origin asymptotically, i.e.,
			$x_k\rightarrow 0,\ \ u_k\rightarrow 0,\   \text{and}\ s_k\rightarrow 0$ asymptotically.
		\end{theorem}
		\textbf{Proof.} Note that  $\hat u$, $\hat s$, $\hat x$ converge to the origin asymptotically as shown in Theorem~\ref{the:rmpc-robust}. Under $w_o=0$, we write~\eqref{Eqn:linear_p-residual} in the following autonomous system:
		\begin{equation}\label{Eqn:linear_p-residual-no-dis}
			\left\{\begin{array}{ll}
				{s}_{k+1}=F{s}_k+F_ds_{k-\tau}+\bar w_k,\\
				\|\bar w_k\|\leq  \|s\|_{\bar L_s}+\|s_{-\tau}\|_{\bar L_{s,d}},\\
				x_k=C{s}_k+v_k.
			\end{array}\right.
		\end{equation}
		Choose a Lyapunov-Krasovskii functional $V_{P,k}=\|{s}_k\|^2_P+\sum_{i=1}^{\tau}\|{s}_{k-i}\|_{S}^2$. The difference of the Lyapunov-Krasovskii functional at adjacent time instants are
		\begin{equation}
			\begin{array}{lll}
				\Delta V_P(k)&=&\|{s}_{k+1}\|^2_P-\|{s}_{k}\|^2_P+\|{s}_{k}\|_{S}^2-\|{s}_{k-\tau}\|_{S}^2\\
				&=&\|s_k\|^2_{{\Pi}_1(P,S)}+\|s_{k-\tau}\|^2_{{\Pi}_2(P,S)}+\|\bar w_k\|_P^2\\
				&&+2s_k^\top F^\top  P F_d s_{-\tau}+2[s_k^{\top}\,s_{k-\tau}^{\top}]^{\top}[F\,F_d]^{\top}P\bar w_k,
			\end{array}
		\end{equation} resulting in~\eqref{Eqn:iss-con} by applying the S-procedure in~\cite{1994Linear}. Hence, $s,u$  converge to the origin asymptotically. Consequently, $x=Cs+v\rightarrow 0$ asymptotically due to $\Psi(x)=0$. \hfill$\blacksquare$
		\section{Adaptive KMPC for time delayed systems}\label{sec:adaptive}
		In this section, we present the robust Koopman MPC with state and control constraints updated online, termed A-KMPC, where uncertainty sets $\mathcal{W}$ and $\mathcal{V}$ are recursively updated to reduce conservativity.
		\subsection{Uncertainty sets updates}
		In view of $\Psi^{-1}(s)=Cs+v(s)$ in Assumption~\ref{assum:invers-phi} and of its Lipschitz property, one has $v(s)=\Psi^{-1}(s)-Cs$, such that 
		\begin{equation}\label{Eqn:inv-phi-lipsch}
				\|v(s_1)-v(s_2)\|\leq L_v \|s_1-s_2\|, 
		\end{equation}
		for all $s_1,\,s_2\in\mathcal{S}_{\Psi}$, where $L_v$ is the Lipschitz constant. 
		In view of~\eqref{Eqn:lipsch-1} and~\eqref{Eqn:inv-phi-lipsch}, by setting $s_1=s$, $u_1=u$, $w_{o,1}=w_o$, $s_2=u_2=w_{o,2}=0$, one has 
		\begin{equation}\label{Eqn:dis-set}
			\left\{\begin{array}{ll}
				\|v\|\leq L_v \|s\|\\
				\|w\|\leq L_s \|s\|+L_{s,d} \|s_{-\tau}\|+L_u \|u\|+L_w \|w_o\|,
			\end{array}\right.
		\end{equation} for all $s\in\mathcal{S}_{\Psi}$, $u\in \mathcal{U}$, and $w_o\in\mathcal{W}_o$. 	 
		Since $w$ and $v$ are parametric uncertainties, we allow both sets  to be updated online along with variations of $s$ and $u$ in each time $l$ belonging to the prediction window $[k,k+N]$, of type:
		\begin{subequations}\label{Eqn:dis-set-adap}
			\begin{align}
				\mathcal{V}^l&=\{v|H_vv_{l}\leq b_{v,l}\}\\
				{\mathcal{W}}^l&=\{w|H_w w_{l}\leq b_{w,l}\},
			\end{align}
			where $H_v=[I_{n}\ -I_n]^{\top}$, $b_{v,l}=(\bar b_{v,l},\bar b_{v,l})$, $\bar b_{v,l}=L_v \bar s_l$, $H_w=[I_{n_{\psi}}\ -I_{n_{\psi}}]^{\top}$, $b_{w,l}=(\bar b_{w,l},\bar b_{w,l})$, {\color{black} $\bar b_{w,l}=(L_s+L_{s,d}) \bar s_l+\bar L_u \bar u_l+L_{w,l}\bar w_o$, where $\bar L_u=\sup_{\scriptscriptstyle Bu\neq 0, u\in\mathcal{U}}L_u \|u\|/\|B u\|$}, 
			$\bar w_o=\arg\max_{w\in\mathcal{W}_o} \|w\|_{\infty}$, $\bar s_l\in\mathbb{R}^{n_{\psi}}$ and $\bar u_l\in\mathbb{R}^{m}$ are optimization variables based on which, $\mathcal{V}^l$ and $\mathcal{W}^l$ are updated due to the following constraints:
			\begin{align}
				&-\bar s_{l} \leq s_{l}\leq \bar s_{l}\label{EqnEqn:dis-set-adap-c}\\
				&-\bar u_{l} \leq B u_{l}\leq  \bar u_{l}\label{EqnEqn:dis-set-adap-d}.
			\end{align}
			Letting $\mathcal{Z}_s^l$ be the updated robust positively invariant set of~\eqref{Eqn:e} with $\bar w_{l}\in\mathcal{W}^l\oplus D\mathcal{W}_o$, we enforce that
			\begin{align}
				&\bar s_{l}\geq h_{3,l}\  \text{ and} \ \bar u_{l}\geq h_{4,l}, \label{Eqn:dis-set-adap-e}
			\end{align}
			where $[h_{4,l}]_i=\max_{e,e_{-\tau}\in\mathcal{Z}_s^l}|[BK]_ie+[BK_d]_ie_{-\tau}|$, $[h_{3,l}]_i=\max_{e\in\mathcal{Z}_s^l}|[e]_i|$. 
		\end{subequations}
		Also, in view of the constraint satisfaction design of robust Koopman MPC (see~\eqref{Eqn:e} and~\eqref{Eqn:constraint-MPC}), 
		for all $l\in[k,k+N-1]$, the variations of $\mathcal{W}^l$, $\mathcal{V}^l$, and $\mathcal{Z}_s^l$, must fulfill the following constraints: 
		\begin{subequations}\label{Eqn:robust-constrant}
			\begin{align}
				&{e}_{k-j}\in \mathcal{Z}_s^k,\, j\in\mathbb{N}_0^{\tau}\quad \label{EQN:ROBUST-CON-A}\\
				&F\mathcal{Z}_s^l\oplus F_d\mathcal{Z}_s^{l-\tau}\oplus\bar{\mathcal{W}^l}\subseteq \mathcal{Z}_s^{l+1} \label{EQN:ROBUST-CON-B}\\
				&Cs_{l}\in \mathcal{X}\ominus \mathcal{V}^{l}\label{EQN:ROBUST-CON-C}\\& u_{l}\in\mathcal{U},\label{EQN:ROBUST-CON-D}
			\end{align}
		\end{subequations}
		where $\bar{\mathcal{W}}^l={\mathcal{W}}^l\oplus D\mathcal{W}_o$.\\
		%
		In the following, we focus on the formulation of time-varying constraints~\eqref{Eqn:robust-constrant} in the online nominal MPC problem, inspired by the work~\cite{lorenzen2019robust}. To proceed, let the initial robust positively invariant set ${\mathcal{Z}}_s^0=\mathcal{Z}_s=\{e|H_ee\leq\bm 1\}$ with vertices $e^1,\cdots,e^{v}$. For all $l\in[k,k+N]$, we introduce  $r_{l}\in \mathbb{R}^{n}$ and a scaling factor $\gamma_{l}\geq 0$, to define $\mathcal{Z}_s^l$ of type
		\begin{equation}\label{Eqn:op-robust-set}
			\begin{array}{ll}
				{\mathcal{Z}}_s^l&={r_{l}}\oplus \gamma_{l}{\mathcal{Z}_s^0}\vspace{1mm}\\
				&=\{e|H_e(e-r_{l})\leq\gamma_{l}\bm 1\},
			\end{array}
		\end{equation}
		where $H_e\in\mathbb{R}^{p\times n_{\psi}}$, $r_{l}$ and $\gamma_{l}$ are decision variables to be optimized in the MPC problem for fulfillment of~\eqref{Eqn:robust-constrant}.\\
		%
		{\color{black}
			To describe the following proposition in a compact form, we herewith introduce some necessary notations and definitions. Let $\bar{\mathcal{S}}_{\Psi}=\{s|Cs\in\mathcal{X}\ominus\mathcal{V}\}=\{s|H_ss\leq \bm 1\}$, $H_s\in\mathbb{R}^{q\times n_{\psi}}$ and $\mathcal{U}=\{u|G_uu\leq \bm 1\}$, $G_u\in\mathbb{R}^{g\times m}$.  For any $i\in\mathbb{N}$, we define  
			$[\tilde e_{h_1}]_i=\max_{j\in\mathbb{N}_1^v}[H_{s}F]_ie^j$, $[\tilde e_{h_1}^{\tau}]_i=\max_{j\in\mathbb{N}_1^v}[H_{s}F_d]_ie_{-\tau}^j$, $[\tilde e_{h_2}]_i=\max_{j\in\mathbb{N}_1^v}[H_{s}]_ie^j$, $[\bar e]_i ([\underline{e}]_i)=\max({\rm min})_{j\in\mathbb{N}_1^v}[e^j]_i$, $[\tilde e_g]_i=\max_{j\in\mathbb{N}_1^v}[G_uK]_ie^j$, $[\tilde e_g^{\tau}]_i=\max_{j\in\mathbb{N}_1^v}[G_uK_d]_ie_{-\tau}^j$, $[\tilde e]_i=\max_{j\in\mathbb{N}_1^v}|[e^j]_i|$, 
			$[\tilde w_{o}]_i=\max_{w\in\mathcal{W}_o}[H_sD]_iw,\break$ $[\overline {e}_b]_i([\underline{e}_b]_i)=\max({\rm min})_{j\in\mathbb{N}_1^v}[BK]_ie^j$, $[\overline {e}_b^{\tau}]_i([\underline{e}_b^{\tau}]_i)=\max({\rm min})_{j\in\mathbb{N}_1^v}[BK_d]_ie_{-\tau}^j$.} 
		The following proposition about the prediction sets is obtained.
		\begin{proposition}[Online sets updates]
			Let $\bm r_{\scriptscriptstyle k:k+N}$, $\bm \gamma_{\scriptscriptstyle k:k+N}$, $\bm {\bar u}_{\scriptscriptstyle k:k+N-1}$, $\bm{\bar s}_{\scriptscriptstyle k:k+N-1}$, $\bm{\hat u}_{\scriptscriptstyle k}$, and $\hat s_{k-j}$, $\forall j\in\mathbb{N}_0^{\tau}$, be the MPC decision variables at the time instant $k$. Constraints~\eqref{Eqn:dis-set-adap} and~\eqref{Eqn:robust-constrant} are satisfied  for all $l\in[k,k+N]$,  iff there exists a matrix ${\Lambda}=[\lambda_1^{\top}\cdots\lambda_q^{\top}]$, $\lambda_i\in\mathbb{R}^{2 n_{\psi}}$, such that
			\begin{subequations}\label{Eqn:adap-constr}
				\begin{align}
					&	H_e(s_{k-j}-\hat s_{k-j}-r_{k-j})\leq \gamma_{k-j}\bm 1,\, j\in\mathbb{N}_0^{\tau}\label{Eqn:constr-a-delay}\\
					&H_s(\hat s_{l+1}+Fr_{l})+ \Lambda b_{w,l}\leq \bm 1-\gamma_{l}\tilde e_{h_1}-\gamma_{l-\tau}\tilde e_{h_1}^{\tau}-\tilde w_{o}\label{Eqn:constr-b}\\
                    & [H_s]_i-\sum_{j=1}^{2n_{\psi}}[\lambda_i]_j[H_I]_j=0,\ i=1,\cdots,q\label{Eqn:constr-c}\\
                    &H_s(\hat s_{l}+r_{l})\leq \bm 1-\gamma_{l}\tilde e_{h_2}\label{Eqn:constr-d}\\
                    &G_u(\hat u_{l}+Kr_{l}+K_d r_{l-\tau})\leq \bm 1- \gamma_{l}\tilde e_g-\gamma_{l-\tau}\tilde e_g^{\tau}\label{Eqn:constr-e}\\
					&-\bar s_{l}- \gamma_{l}\underline{e} \leq \hat s_{l}+r_{l}\leq \bar s_{l}-\gamma_{l}\bar e\label{Eqn:constr-f}\\
                    &-\bar u_{l}-\underline{e}_{L}\leq B(\hat u_{l}+K r_{l}+K_d r_{l-\tau})\leq  \bar u_{l}+\overline{e}_{L}\label{Eqn:constr-g}\\
					&\bar s_{l}\geq h_{3,l}\  \text{ and} \ \bar u_{l}\geq h_{4,l}, \label{Eqn:dis-set-adap-e-final}
				\end{align}
			\end{subequations}
			where $\underline{e}_{L}=\gamma_{l}\underline{e}_b+\gamma_{l-\tau}\underline{e}_b^{\tau}$, $\overline{e}_{L}=\gamma_{l}\overline {e}_b + \gamma_{l-\tau}\overline {e}_b^{\tau}$, $h_{3,l}$ and $h_{4,l}$ are rewritten as  $h_{3,l}=r_{l}+\gamma_{l}\tilde e$, $h_{4,l}=BKr_{l}+BK_dr_{l-\tau}+\gamma_{l}\bar e_b+\gamma_{l-\tau}\bar e_b^{\tau}$.
		\end{proposition}
                    		\begin{algorithm}[h]
			\caption{Pseudocode of adaptive Koopman MPC}
			\label{alg:ra-KMPC}
			\textbf{Off-line designs:}
			\begin{algorithmic}[1]
				\State Select $\Psi(x)$ such that Assumptions~\ref{assum:invers-phi} and~\ref{assu:phi-ori} are verified.
				\State Compute $A$, $A_d$, $B$, $C$, and $D$ with~\eqref{Eqn:appro_K-delay}, and check that Assumption~\ref{assum:stabi-obser} is verified.
				\State Calculate $\bar{\mathcal{W}}$ and $\mathcal{V}$ according to~\cite{zhang2022robust}.
				\State Compute $K,\, K_d$, and  $\bar P$ with~\eqref{Eqn:LYA-mod} and calculate the robust positively invariant set $\mathcal{Z}_s$, $\mathcal{Z}_x$ with $K,\,K_d$ {\color{black}according to~\cite{mayne2005robust}}.
				\State Compute  $\mathcal{S}$, $\hat {\mathcal{U}}$; calculate the terminal set $\mathcal{S}_f$ with $K,\, K_d$.
				\State Calculate $L_v$, $L_s$, $L_{s,d}$, $L_u$, and $L_w$ according to~\cite{zhang2022robust}.
				\State Compute $\bar L_u$,  $\tilde e_{h_1}$, $\tilde e_{h_2}$, $\underline{e}$, $\bar{e}$, $\underline{e}_b$, $\overline{e}_b$, $\tilde e$, and $\tilde e_b$ to construct constraints~\eqref{Eqn:adap-constr}.
			\end{algorithmic}
			\dotfill \\
			\textbf{On-line procedures:}\\
			At each discrete-time step $k=1,2,\cdots$
			\begin{algorithmic}[1]
				\State Measure $x_k,\,x_{k-\tau}$ and set the lifted state $s_k=\Psi(x_k)$ and $s_{k-\tau}=\Psi(x_{k-\tau})$.
				\State  Solve~\eqref{Eqn:optimiz1-ad} with~\eqref{Eqn:cost-mod} and obtain ${\bm\hat u}^{\ast}_{k|k}$, $\hat s_{k-j|k}$, $j\in\mathbb{N}_0^{\tau}$.
				\State Set $u_k$ with~\eqref{Eqn:realu} and apply it to the nonlinear system~\eqref{Eqn:non-model}.
			\end{algorithmic}
		\end{algorithm}
		{\color{black}
			\textbf{Proof}. Constraint~\eqref{EQN:ROBUST-CON-A} is equivalent to \vspace{-2mm}
			\begin{equation*}
				\begin{array}{ll}
					H_e(s_{k-j}-\hat s_{k-j}-r_{k-j})\leq \gamma_{k-j}\bm 1,\, j\in\mathbb{N}_0^{\tau},
				\end{array}
			\end{equation*} that is~\eqref{Eqn:constr-a-delay}.
			Likewise,~\eqref{EQN:ROBUST-CON-B} is equivalent to 
			\begin{equation*}
				A s_{l}+A_ds_{l-\tau}+Bu_{l}+\bar w_{l}\in\bar{\mathcal{S}}_{\Psi},\ \forall \,\bar w_{l}\in\bar{\mathcal{W}}^l,\, s_{l},s_{l-\tau}\in \bar{\mathcal{S}}_{\Psi}.
			\end{equation*} 
			In view of $u=\hat u+K(s-\hat s)+K_d(s_{-\tau}-\hat s_{-\tau})$, $s=\hat s+e$, $s_{-\tau}=\hat s_{-\tau}+e_{-\tau}$, the above constraint is equivalent to\vspace{-2mm}
			\begin{equation*}
				\max_{e_l\in\mathcal{Z}_s^l,e_{l-\tau}\in\mathcal{Z}_s^{l-\tau}}H_s(\hat s_{l+1}+Fe_{l}+F_de_{l-\tau}+\bar w_{l})\leq \bm 1,\ \forall \,\bar w\in\bar{\mathcal{W}}^l,
			\end{equation*} 
			\vspace{-4mm}
			\begin{equation*}
				\Leftrightarrow\max_{\bar w\in\bar{\mathcal{W}^l}} H_s(\hat s_{l+1}+Fr_{l}+F_dr_{l-\tau}+\bar w_{l})\leq \bm 1-\gamma_{l}\tilde e_{h_1}-\gamma_{l-\tau}\tilde e_{h_1}^{\tau},
			\end{equation*} due to $e^j_{l}=r_{l}+\gamma_{l} e^j$ in view of~\eqref{Eqn:op-robust-set}. It is equivalent to\vspace{-2mm}
			\begin{equation*}
				\begin{array}{ll}
					\max_{ w\in{\mathcal{W}^l}} H_s(\hat s_{l+1}+Fr_l+F_dr_{l-\tau}+ w_{l})\leq\\ \hspace{35mm}\bm 1-\gamma_{l}\tilde e_{h_1}-\gamma_{l-\tau}\tilde e_{h_1}^{\tau}-\tilde w_{o},
				\end{array}
			\end{equation*}
			where by duality, for each $i\in\mathbb{N}_1^q$, $\max [H_s]_iw$ subject to $H_Iw\leq b_w$, is equivalent to \vspace{-2mm}
			\begin{equation*}
				\min{\lambda_i} \max_w [H_s]_iw+\sum_{j=1}^{2n_{\psi}}[\lambda_i]_j([b_w]_j-[H_I]_jw),
			\end{equation*}
            \vspace{-3mm}
			\begin{equation*}
				\begin{array}{ll}
					\Leftrightarrow & \min{\lambda_i} \sum_{j=1}^{2 n_{\psi}}[\lambda_i]_j[b_w]_j,\vspace{1mm}\\
					&\text{s.t.:}\quad [H_s]_i-\sum_{j=1}^{2n_{\psi}}[\lambda_i]_j[H_I]_j=0, \ [\lambda_i]_j\geq 0.
				\end{array}
			\end{equation*}
			Hence,~\eqref{EQN:ROBUST-CON-B} is equivalent to~\eqref{Eqn:constr-b} and~\eqref{Eqn:constr-c}. 
			\eqref{EQN:ROBUST-CON-C} is equivalent to $\max_{e\in\mathcal{Z}_s^l}H_s(\hat s_{l}+e_{l})\leq \bm 1$, that is~\eqref{Eqn:constr-d}. Similarly, \eqref{EQN:ROBUST-CON-D} is equivalent to
			$\max_{e\in{\mathcal{Z}_s^l}}G_u(\hat u_{l}+Ke_{l}+K_de_{l-\tau})\leq\bm 1$, i.e., constraint~\eqref{Eqn:constr-e}. Also, \eqref{EqnEqn:dis-set-adap-c} is equivalent to $-\bar s_{l} \leq \hat s_{l}+r_{l}+\gamma_{l}e\leq \bar s_{l},\, \forall e\in\mathcal{Z}_s^0$, and to~\eqref{Eqn:constr-f}  via maximize (minimize) over $e\in\mathcal{Z}_s^0$ to the right (left) hand side of the inequality. Likewise,~\eqref{EqnEqn:dis-set-adap-d} is equivalent to $-\bar u_{l} \leq B (\hat u_{l}+Kr_{l}+K\gamma_{l}e+K_dr_{l-\tau}+K_d\gamma_{l-\tau}e_{-\tau})\leq  \bar u_{l},\, \forall e\in\mathcal{Z}_s^0$, and to~\eqref{Eqn:constr-g}. \eqref{Eqn:dis-set-adap-e} is equivalent to~\eqref{Eqn:dis-set-adap-e-final} in view of~\eqref{Eqn:op-robust-set}.}
		\hfill$\blacksquare$
		
		\subsection{Robust Koopman MPC with sets updates}\label{sec:r-KMPC-updata}
		The proposed controller here is similar to~\eqref{Eqn:real_u}. The main difference lies in a new formulation of the nominal MPC~\eqref{Eqn:optimiz1} with constraints updates.
		At any time instant $k$, an online optimization problem is to be solved:
		{\color{black}
			\begin{equation}\label{Eqn:optimiz1-ad}
				\min{\{\hat s_{k-j}\}_{j=0}^{\tau},{\bm {\hat u}}_k,\bm r_k,\bm \gamma_k,\bm{\bar s}_k,\bm{\bar u}_k, \Lambda}  { \bar V\big(\hat{s}_{k},\bm {\hat u}_k,\bm{\bar s}_k,\bm{\bar u}_k\big)},
			\end{equation}
			where
			\begin{equation}\label{Eqn:cost-mod}
				\begin{array}{ll}
					\bar V=\bar V_f(\hat{s}_{k+N})+\sum_{i=0}^{N-1}(\|\hat {x}_{k+i}\|_{ Q}^2+\|\hat u_{k+i}\|_{R}^2+\vspace{1mm}\\
					q_s\|\bar s_{k+i}-h_{3,k+i}\|^2
					+q_u\|\bar u_{k+i}-h_{4,k+i}\|^2),
					
				\end{array}
			\end{equation}
			and $q_s>0$ and $q_u>0$ are tuning scalars, $\bar V_f(\hat s)$ is defined as
			$\bar V_f(\hat{s})=\|\hat{s}\|^2_{\bar P}+\sum_{i=1}^{\tau}\|\hat{s}_{-i}\|_{\bar S}^2,$
			where the symmetric positive-definite matrices $\bar P$ and $\bar S$ satisfy the following inequality:
			\begin{equation}\label{Eqn:LYA-mod}
				\begin{bmatrix}\bar{\Pi}_1(\bar P,\bar S)&\star\\
					F_d^{\top}\bar PF+K_d^{\top}(R+q_uB^{\top}B)K&\bar{\Pi}_2(\bar P,\bar S)
				\end{bmatrix}\preceq 0,
			\end{equation}
			where $\bar{\Pi}_1(\bar P,\bar S)={\Pi}_1(\bar P,\bar S)+q_s+q_uK^{\top}B^{\top}BK$, $\bar{\Pi}_2(\bar P,\bar S)\break={\Pi}_2(\bar P,\bar S)+q_uK_d^{\top}B^{\top}BK_d$.}
		
		The above optimization problem is performed 
		subject to the following constraints:\vspace{-2mm}
		\begin{enumerate}[1)]
			\item The nominal Koopman model~\eqref{Eqn:unpert};
			\item The constraint~\eqref{Eqn:adap-constr};
			\begin{subequations}
				\item The terminal state constraint 
				\begin{align}
					\hat {s}_{k+N-j}\in {{\mathcal{S}}_f},\,\forall j\in\mathbb{N}_0^{\tau}.
				\end{align}
			\end{subequations}
		\end{enumerate} 
		Problem~\eqref{Eqn:optimiz1-ad} is repeatedly solved according to the receding horizon principle. The final control applied to system~\eqref{Eqn:non-model} is given as
		\begin{equation}\label{Eqn:realu}
			u_{k}=\hat u^o_{k|k}+K({s}_{k}-\hat{s}_{k|k})+K_d({s}_{k-\tau}-\hat{s}_{k-\tau|k}).
		\end{equation}
		\begin{theorem}[Recursive feasibility]\label{the:the-recur-dap-mpc} Under Assumption \ref{assu:phi-ori}, if problem~\eqref{Eqn:optimiz1-ad} with~\eqref{Eqn:cost-mod} is feasible at time $k=0$, then it is recursively feasible for all $k\geq 1$.
		\end{theorem}
		\textbf{Proof}. We assume~\eqref{Eqn:optimiz1-ad} with~\eqref{Eqn:cost-mod} is feasible at time $k$, i.e., one can find the optimal solution, $\hat s^{\ast}_{k-j|k}$, ${\bm {\hat u}}^{\ast}_{\scriptscriptstyle k|k}$, $\bm r^{\ast}_{\scriptscriptstyle k|k}$, $\bm \gamma^{\ast}_{\scriptscriptstyle k|k}$, $\bm{\bar s}^{\ast}_{\scriptscriptstyle k|k}$, $\bm{\bar u}^{\ast}_{\scriptscriptstyle k|k}$, $\Lambda^{\ast}$, $j\in\mathbb{N}_0^{\tau}$, such that~\eqref{Eqn:dis-set-adap} and~\eqref{Eqn:robust-constrant} are satisfied. Let in the subsequent time instant a candidate solution be\vspace{-3mm}
		\begin{equation}\label{Eqn:feasible-solu0}
			\begin{array}{ll}
				\Lambda^{s}=\Lambda^{\ast},\ \hat s^s_{k+1-j}=\hat s_{k+1-j|k},\, j\in\mathbb{N}_0^{\tau}\\
				{\bm {\hat u}}^{s}_{\scriptscriptstyle k+1}={\bm{\hat u}}^{\ast}_{\scriptscriptstyle k+1:k+N-1|k}, K\hat s_{k+N|k}+K_d\hat s_{k+N-\tau|k}\\
				\bm r^{s}_{\scriptscriptstyle k+1}={\bm{r}}^{\ast}_{\scriptscriptstyle k+1:k+N|k},0\\
				\bm \gamma^{s}_{\scriptscriptstyle k+1}={\bm{\gamma}}^{\ast}_{\scriptscriptstyle k+1:k+N|k},1,\vspace{-2mm}
			\end{array}
		\end{equation}
		and\vspace{-2mm}
		\begin{equation}\label{Eqn:feasible-solu}
			\begin{array}{ll}
				\bm {\bar s}^s_{\scriptscriptstyle k+1}=\bm {\bar s}^{\ast}_{\scriptscriptstyle k+1:k+N-1|k},\bar s^t,\,
				\bm {\bar u}^s_{\scriptscriptstyle k+1}=\bm {\bar u}^{\ast}_{\scriptscriptstyle k+1:k+N-1|k},\bar u^t,
			\end{array}
		\end{equation}
		where 
			$\bar s^t=\bar{\hat s}_{k+N|k}+h_{3,k+N|k},\,
			\bar u^t=\bar{\hat u}_{k+N|k}+h_{4,k+N|k},$
		and $[\bar{\hat s}]_i=|[{\hat s}]_i|$, $i=1,\cdots, n_{\psi}$, $[\bar{\hat u}]_i=|[BK]_i\hat s|+|[BK_d]_i\hat s_{-\tau}|$, $i=1,\cdots, m$.
		One can conclude by inheritance that constraint~\eqref{Eqn:dis-set-adap} is satisfied and\vspace{-2mm}
		\begin{equation*}
			\begin{array}{ll}
				&{s}_{k+1-j}-\hat{s}_{k+1-j|k+1}\in \mathcal{Z}_s^k,\,j\in\mathbb{N}_0^{\tau}\\
				&F\mathcal{Z}_s^{\scriptscriptstyle l|k+1}\oplus F_d\mathcal{Z}_s^{\scriptscriptstyle l-\tau|k+1}\oplus\bar{\mathcal{W}}^{\scriptscriptstyle l|k+1}\subseteq \mathcal{Z}_s^{\scriptscriptstyle l|k+1}\\
				&x_{l|k+1}\in\mathcal{X},\, u_{l|k+1}\in\mathcal{U},\vspace{-1mm}
			\end{array}
		\end{equation*}	
		for all $ l=k+1,\cdots,k+N$.  Moreover, 
		$\hat {s}_{k+N+1-j|k+1}\in {{\mathcal{S}}_f}$, $j\in\mathbb{N}_0^{\tau}$, by the choice of $ K\hat s_{k+N|k}+K_d\hat s_{k+N-\tau|k}$ in $\bm u^s$, $\mathcal{Z}_s^{\scriptscriptstyle k+N+1|k+1}=\mathcal{Z}_s^0$ by $\bm r^{s}$ and $\bm \gamma^{s}$. Hence, one promptly has $s_{k+N+1|k+1}\in{\mathcal{S}}_f\oplus\mathcal{Z}_s$, leading to the satisfactions of $x_{k+N+1}\in\mathcal{X}$, $u_{k+N}\in\mathcal{U}$. Also~\eqref{EqnEqn:dis-set-adap-c},~\eqref{EqnEqn:dis-set-adap-d}, and \eqref{Eqn:dis-set-adap-e} are fulfilled at the terminal time instant by~\eqref{Eqn:feasible-solu}.   Hence, the recursive feasibility is verified.
		\hfill $\blacksquare$\vspace{-1mm}
		\begin{theorem}[Closed-loop robustness] Under Assumptions~\ref{assum:invers-phi} and \ref{assu:phi-ori}, the lifted nominal system~\eqref{Eqn:unpert} using $\hat s_{k},\hat u_{k}$ computed with~\eqref{Eqn:optimiz1-ad} is asymptotically stable.
			Also, the states of model~\eqref{Eqn:linear_p-residual} and nonlinear system~\eqref{Eqn:non-model} with control~\eqref{Eqn:real_u} converge asymptotically to the sets ${\mathcal{Z}}_s^k$ and ${\mathcal{Z}}_x^k=C{\mathcal{Z}}_s^k\oplus\mathcal{V}^k$, i.e.,
			$s_{k}\rightarrow{\mathcal{Z}}_s^k,\ x_{k}\rightarrow {\mathcal{Z}}_x^k\ \text{as}\ k\rightarrow+\infty.$
		\end{theorem}\vspace{-1mm}
		{\color{black}
			\textbf{Proof}. Inline with the arguments in Theorem 2, letting $\bar V_k$ be the sub-optimal cost under~\eqref{Eqn:feasible-solu0} and~\eqref{Eqn:feasible-solu}, one has \vspace{-1mm}
			\begin{equation*}
				\begin{array}{lll}
					&\bar V^{\ast}_{k+1}-\bar V^{\ast}_{k}\leq\bar V^{s}_{k+1}-\bar V^{\ast}_{k}\\
					&\leq\hspace{-1mm}-\|\hat {x}_{k}\|_{Q}^2-\|\hat u_{k}\|_{R}^2\vspace{1mm}-
					q_s\|\bar s_{k}-h_{3,k}\|^2-q_u\|\bar u_{k}-h_{4,k}\|^2\\
					&\hspace{2mm}+\bar V_f(\hat{s}_{k+N+1})-\bar V_f(\hat{s}_{k+N})+q_s\|\bar{\hat s}_{k+N}\|^2+q_u\|\bar{\hat u}_{k+N}\|^2.\vspace{-1mm}
				\end{array}
			\end{equation*}
			Taking into account that $\|\bar{\hat u}_{k+N}\|^2=\|BK{\hat s}_{k+N}+BK_d{\hat s}_{k+N-\tau}\|^2,$ 	in view of~\eqref{Eqn:LYA-mod}, one has
					$\bar V_f(\hat{s}_{k+N+1})-\bar V_f(\hat{s}_{k+N})+q_s\|{\hat s}_{k+N}\|^2+
					q_u\|BK{\hat s}_{k+N}+BK_d{\hat s}_{k+N-\tau}\|^2\leq 0.$
			Hence, $\bar V^{\ast}_{k+1}-\bar V^{\ast}_{k}\leq-\|\hat {x}_{k}\|_{Q}^2-\|\hat u_{k}\|_{R}^2-q_s\|\bar s_{k}-h_{3,k}\|^2-q_u\|\bar u_{k}-h_{4,k}\|^2,$ 
			leading to the asymptotic stability of \eqref{Eqn:unpert}, i.e., $\hat s_{k}$, $\hat x_{k}$, and $\hat u_{k}$ to the origin, and of variables $\bar s_{k}$, $\bar u_{k}$ to $h_{3,k}$, $h_{4,k}$ respectively, see Theorem 2.}\\
		As for the closed-loop robustness, one can conclude inline with Theorem 2 that $s_{k}\rightarrow{\mathcal{Z}}_s^k,\ x_{k}\rightarrow {\mathcal{Z}}_x^k$ asymptotically. 
		\hfill $\blacksquare$
       {\color{black} \begin{remark}
        Note that, the results $\bar s_{k+i}\rightarrow h_{3,k+i},
\,
\bar u_{k+i}\rightarrow h_{4,k+i},$ as $k\rightarrow +\infty$ imply that the RPI set can be automatically reduced to a relatively small region. However, this region may not be unique due to the online time-varying uncertainty adaptation, and it may not be minimal because of the conservatism introduced by constraint~\eqref{Eqn:dis-set-adap-e}.
        \end{remark}}
		Let $\tilde{\bar \Pi}(\bar P,\bar S)$ be defined in line with $\tilde{\Pi}( P, S)$ in~\eqref{Eqn:iss-con}. The following corollary is stated.
		\begin{corollary}[Point-wise convergence]
			If $w_o=0$ and $\bar{\tilde \Pi}(\bar P,\bar S)\preceq 0$ is satisfied,
			the closed-loop lifted Koopman model~\eqref{Eqn:linear_p-residual} and nonlinear one~\eqref{Eqn:non-model} with control~\eqref{Eqn:real_u} under~\eqref{Eqn:optimiz1-ad} converge to the origin asymptotically, i.e.,
			$$x_{k}\rightarrow 0,\ \ u_{k}\rightarrow 0,\   \text{and}\ s_{k}\rightarrow 0\ \text{as}\ k\rightarrow+\infty.$$
		\end{corollary}\vspace{-2mm}
        		\begin{figure}[!t]
			\centering
			\includegraphics[width=2.2in]{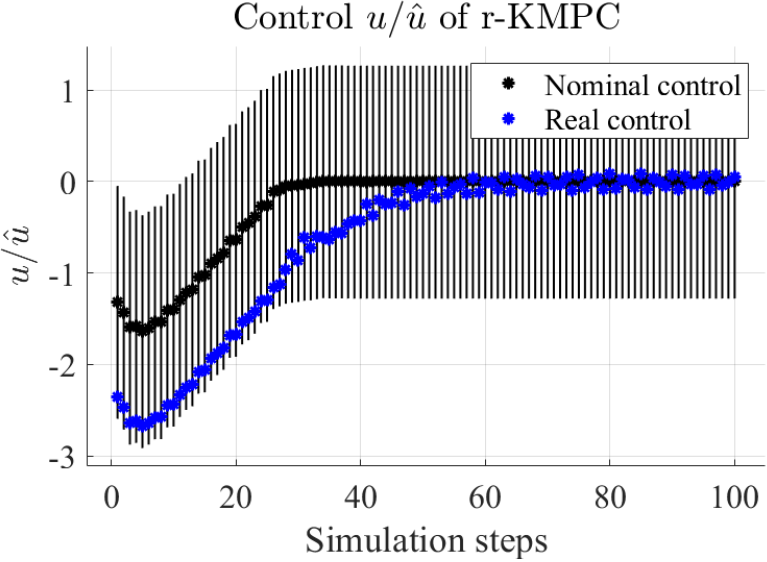}
			\caption{The control tubes of the time delayed inverted pendulum with r-KMPC.\vspace{-2mm}}
			\label{Fig: u-tube of IP with r-KMPC}
		\end{figure}
		
		\begin{figure}[!t]
			\centering
			\includegraphics[width=2.2in]{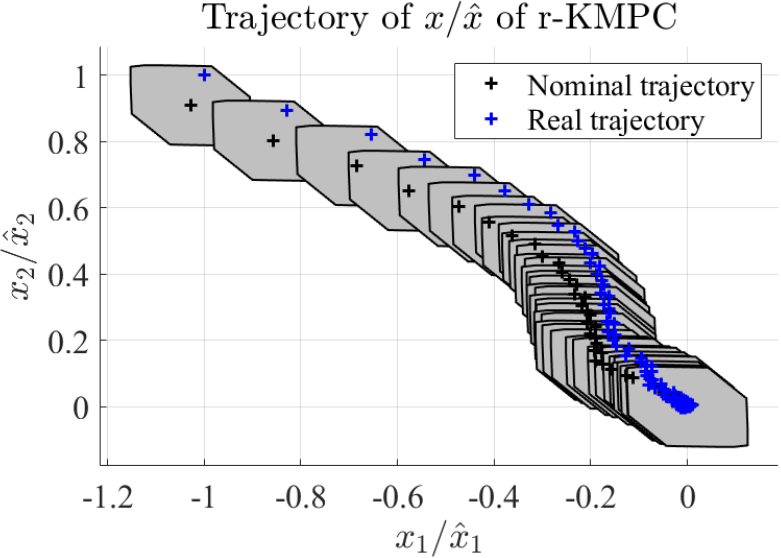}
			\caption{The real and nominal trajectory of the time delayed inverted pendulum with r-KMPC.\vspace{-2mm}}
			\label{Fig: x-tube of IP with r-KMPC}
		\end{figure}
		%
		
		\begin{figure}[!t]
			\centering
			\includegraphics[width=2.4in]{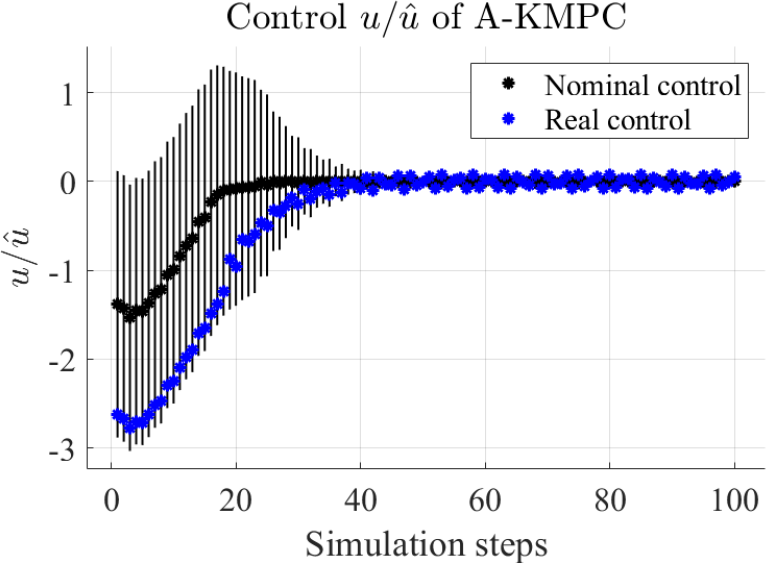}
			\caption{The control tubes of the time delayed inverted pendulum with A-KMPC.\vspace{-2mm}}
			\label{Fig: u-tube of IP with adaptive r-KMPC}
		\end{figure}
		\begin{figure}[!t]
			\centering
			\includegraphics[width=2.2in]{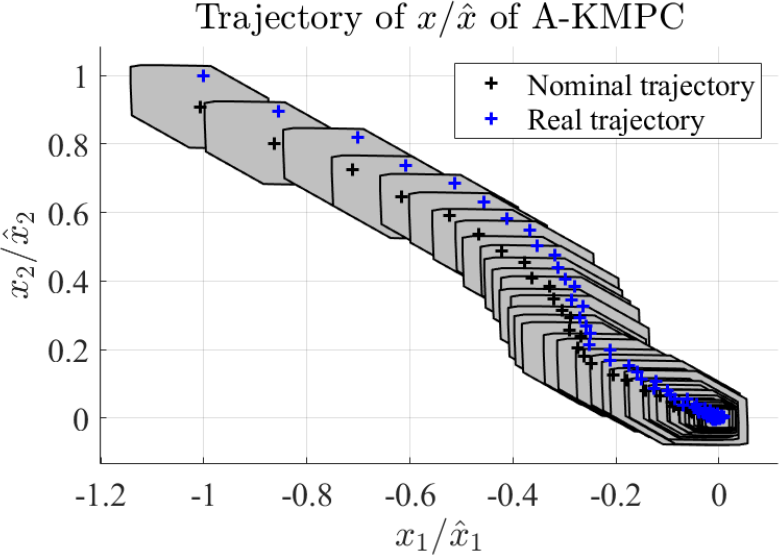}
			\caption{The real and nominal trajectory of the time-delay inverted pendulum with A-KMPC.}
			\label{Fig: x-tube of IP with adaptive r-KMPC}
		\end{figure}
		\section{Simulation results}\label{sec:simulation}
		\subsection{Control of a time delayed inverted pendulum}
		Consider the control of an inverted pendulum with state time delays, whose discrete-time model is presented as
		\begin{equation}
			\left[ \begin{array}{c}
				x_1^+\\
				x_2^+\\
			\end{array} \right] =\left[ \begin{array}{c}
				x_1 +\frac{3g}{4l}T\sin \left( x_{2,-\tau} \right) +\frac{3}{4ml^2}Tu\\
				x_2 +Tx_1\\
			\end{array} \right] +w_o,
		\end{equation}
		where $x_1$ and $x_2$ are the angular velocity and angle of the pendulum, $u$ is the angular acceleration, $\tau = 2$ is the delay step, the sampling time is $T=0.1s$, the state and control constraints are $x_1\in [-2\pi,2\pi]$, $x_2\in [-\pi,\pi]$, $u\in [-5,5]$. The mass of the pendulum is $m = \frac{1}{3} kg$, the length of the pendulum is $l = 1.5 \,m$, the gravity acceleration is $g = 9.8 \,m/s^2$, and $w_o=0.01 \sin({10t})$.\\
		We validated the proposed r-KMPC and A-KMPC algorithms successively. In both scenarios, we chose $\Psi(x)=(x_1,x_2,\sin(x_2))$. The weighting matrices $Q$ and $R$ were selected as $Q=\text{diag}\{1,1,0.1\}$, $R=0.1$. The prediction horizon was set as $N=20$. The simulation test has been performed in Matlab 2019a environment on a Laptop with Intel Core i9-9880H 2.30-GHz. The simulations were implemented with an initial condition $x_0=(-1,1)\in\mathcal{X}$. 
        		\begin{table}[!t]
			\caption{Performance comparison of time delayed inverted pendulum with A-KMPC and r-KMPC.}
			\label{Tab:Comparison of IP}
			\centering
            \scalebox{0.8}{
			\begin{tabular}{ccccc}
				\hline
				\multicolumn{2}{c}{Performance} & $J$ & $J_u$ & $J_x$\\
				\hline
				\multirow{4}{*}{A-KMPC}& $w_o=0$ & 18.07 & 93.56 & 8.71\\
				& $w_o=0.01\sin(10k)$ & 18.01 & 93.20 & 8.69\\
				& $w_o=0.01\cdot(\text{rands}(1))$ & 18.29 & 95.16 & 8.78\\
				& $w_o=0.05\sin(10k)$ & 18.35 & 96.47 & 8.70\\
				\hline
				\multirow{4}{*}{r-KMPC}& $w_o=0$ & 21.97 & 124.27 & 9.55\\
				& $w_o=0.01\sin(10k)$ & 21.88 & 123.69 & 9.51\\
				& $w_o=0.01\cdot(\text{rands}(1))$ & 20.15 & 111.81 & 8.97\\
				& $w_o=0.05\sin(10k)$ & 22.14 & 126.57 & 9.48\\
				\hline
			\end{tabular}}
		\end{table}
    		In implementing r-KMPC, the terminal matrices $P$ and $S$ and the feedback control matrices $K$ and $K_d$ were calculated by \eqref{Eqn:LYA}. The simulation results are shown in Figs. \ref{Fig: u-tube of IP with r-KMPC}-\ref{Fig: x-tube of IP with r-KMPC}, which shows that the nominal lifted state $\hat{s}$ and control $\hat{u}$ converge to the origin, while the real state $x$ and control $u$ converge to the neighborhood of the origin due to the influence of external disturbances. Also, r-KMPC ensures the satisfaction of the state and control constraints, and the robustness of the closed-loop system.\\		
		In implementing A-KMPC, we set $q_s=q_u=0.05$. The terminal matrices $\bar{P}$ and $\bar{S}$ were calculated by \eqref{Eqn:LYA-mod}. All other parameters were calculated according to Algorithm~\ref{alg:ra-KMPC}. The simulation results are shown in Figs. \ref{Fig: u-tube of IP with adaptive r-KMPC}-\ref{Fig: x-tube of IP with adaptive r-KMPC}, which shows that the nominal state $\hat{x}$ and control $\hat{u}$ converge to the origin, while the real state $x$ and control $u$ always satisfy the constraints. The numerical comparison between A-KMPC and r-KMPC is provided in Table \ref{Tab:Comparison of IP}. A-KMPC achieves a substantial reduction in cumulative cost compared to r-KMPC, with decreases of $17.7\%$, $24.7\%$, and $8.67\%$ in $J $, $J_u $, and $J_x$, respectively. In summary, A-KMPC with adaptive sets updates effectively reduces conservatism and enhances control performance. {\color{black}However, the performance improvement comes at the cost of increased computation time (see Table~\ref{Tab:Comparison of IP}). Nevertheless, the reported 19.67 ms for A-KMPC with $N=10$ remains acceptable for many robotic control applications, such as autonomous vehicles with typical sampling intervals of 20--50 ms. Further speedup is possible with optimized C/C++ implementations.}
         \begin{table}[H]
			\caption{Comparison of average computational time (ms).}
			\label{Tab:Comparison of IP}
			\centering
            \scalebox{0.9}{
			\begin{tabular}{ccccc}
				\hline
				\multicolumn{1}{c}{Methods} & $N=5$ & $N=8$ & $N=10$\\
				\hline
				\multirow{1}{*}{A-KMPC}
                & 10.78& 14.82& 19.67\\
				\hline
				\multirow{1}{*}{r-KMPC}
                & 4.33& 4.71& 4.96\\
				\hline
			\end{tabular}}
		\end{table}
		%
		%
		%
		%
        \vspace{-4mm}
		\subsection{Control of a time delayed chemical reactor}
		The chemical reactor is also used to evaluate the effectiveness of the proposed r-KMPC and A-KMPC algorithms, whose discrete-time model is
		\begin{equation}
			\begin{aligned}
				x_1^+= & \left[1-T\left(\frac{1}{\theta_1}+a_1\right)\right] x_1+T \frac{\left(1-R_2\right)}{\varpi_1} x_2\\
				&+T\varrho \sin (kT) x^2_{2,-\tau}+w_{o1}  \\
				x_2^+= & \left[1-T\left(\frac{1}{\theta_2}+a_2\right)\right] x_2+T \frac{R_1}{\varpi_2} x_{1,-\tau}+T \frac{R_2}{\varpi_2} x_{2,-\tau} \\
				&+T\varrho \sin (kT) x_{1,-\tau}^2+T \frac{G_1}{\varpi_2} u+w_{o2}.
			\end{aligned}
		\end{equation}
		where the sampling time is $T=0.1$, the delay step is $\tau=5$, $\theta_1=\theta_2=1, a_1=a_2=1, \varpi_1=\varpi_2=1, R_1=R_2=0.5, G_1=G_2=0.5$, $\varrho=1$, the state and control constraints are $x_1\in [-2,2]$, $x_2\in [-2,2]$, $u\in [-1,1]$, and $w_{o1}=w_{o2}=0.05x_2(t)$.\\
        		\begin{table}[!t]
			\caption{Performance comparison of time delayed chemical reactor with A-KMPC and r-KMPC.}
			\label{Tab:Comparison of CR}
			\centering
            \scalebox{0.9}{
			\begin{tabular}{ccccc}
				\hline
				\multicolumn{2}{c}{Performance} & $J$ & $J_u$ & $J_x$\\
				\hline
				\multirow{4}{*}{A-KMPC}& $w_o=0$ & 4.55 & 4.75 & 4.08\\
				& $w_o=0.05x_2$ & 6.26 & 8.06 & 5.45\\
				& $w_o=0.01\cdot(\text{rands}(1))$ & 4.63 & 4.93 & 4.14\\
				& $w_o=0.01\sin(10k)$ & 4.55 & 4.75 & 4.08\\
				\hline
				\multirow{4}{*}{r-KMPC}& $w_o=0$ & 4.59 & 4.11 & 4.18\\
				& $w_o=0.05x_2$ & 6.41 & 7.26 & 5.68\\
				& $w_o=0.01\cdot(\text{rands}(1))$ & 4.67 & 4.27 & 4.24\\
				& $w_o=0.01\sin(10k)$ & 4.59 & 4.12 & 4.18\\
				\hline
			\end{tabular}}
		\end{table}
		To learn the Koopman model of the chemical reactor, the lifted function was constructed as $\Psi(x)=(x_1,x_2,x_1^2,x_2^2,x_1x_2)$. The weighting matrices $Q$ and $R$ were selected as $Q=\text{diag}\{1,1,0.1,0.1,0.1\}$ and $R=0.1$. The prediction horizon was set to $N=11$. The initial condition is $x_0=(-1.2,1.2)$ and we set $q_s=q_u=0.2$.\\
		The simulation results are presented in Figs.~\ref{Fig: x-tube of CR with r-KMPC}-\ref{Fig: x-tube of CR with adaptive r-KMPC}, demonstrating that both r-KMPC and A-KMPC ensure convergence of the nominal system, constraint satisfaction, and closed-loop robustness. Furthermore, the performance comparison in Table~\ref{Tab:Comparison of CR} indicates that A-KMPC outperforms r-KMPC. This improvement is attributed to A-KMPC’s ability to update state and control constraints online, thereby reducing conservatism and enhancing control performance. 
		
		%
		%
		%
           \begin{figure}[!t]
			\centering
			\includegraphics[width=2.2in]{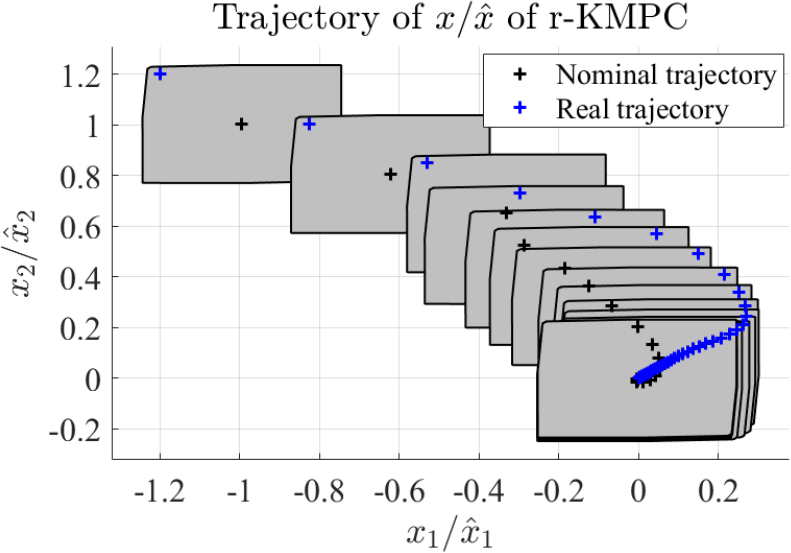}
			\caption{The real and nominal trajectory of the time delayed chemical reactor with r-KMPC.\vspace{-1mm}}
			\label{Fig: x-tube of CR with r-KMPC}
		\end{figure}
            \begin{figure}[!t]
			\centering
			\includegraphics[width=2.2in]{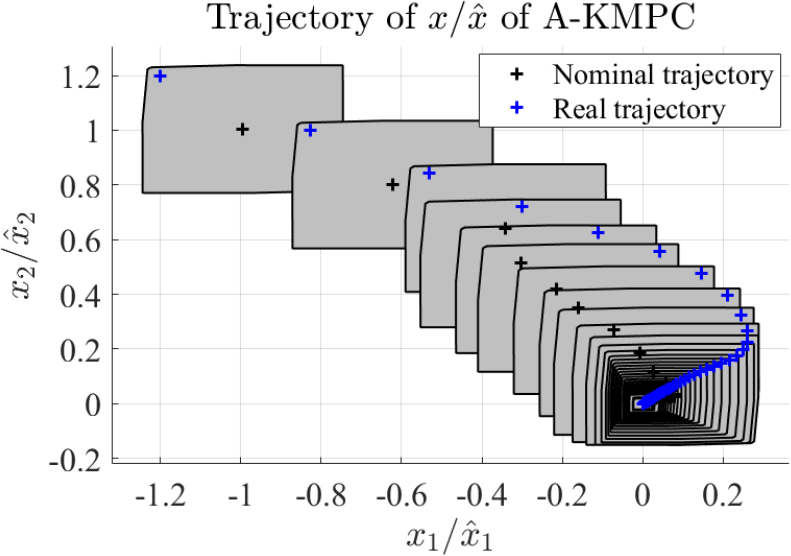}
			\caption{The real and nominal trajectory of the time delayed chemical reactor with A-KMPC.\vspace{-1mm}}
			\label{Fig: x-tube of CR with adaptive r-KMPC}
		\end{figure}
		\section{Conclusions}\label{sec:con}
		In this paper, we proposed a robust adaptive MPC scheme leveraging Koopman operators for nonlinear time delayed systems. By constructing a lifted time delayed Koopman model in the feature space, we first developed an r-KMPC algorithm that guarantees recursive feasibility and closed-loop robustness. Building upon this framework, we further introduced an adaptive robust Koopman MPC variant (A-KMPC), which incorporates online updates of uncertainty sets to reduce the conservatism typically associated with robust MPC designs. We theoretically established closed-loop robustness under exogenous disturbances and asymptotic convergence under nominal conditions. Through two numerical case studies, we demonstrate that both r-KMPC and A-KMPC ensure robust performance. Notably, A-KMPC has achieved significant improvements in control performance and reduced conservatism via real-time set adaptation compared to r-KMPC.
		
		
		\bibliographystyle{unsrt}
		\bibliography{refletter}
		
		
		

	\end{document}